\documentclass[a4paper,11pt]{amsart}

\usepackage{hyperref}
\usepackage{xcolor}
\hypersetup{
    colorlinks,
    linkcolor={red!50!black},
    citecolor={blue!50!black},
    urlcolor={blue!80!black}
}

\usepackage{microtype}
\usepackage{MnSymbol}

\makeatletter
\newcommand{\eqnum}{\refstepcounter{equation}\textup{\tagform@{\theequation}}}
\makeatother

\makeatletter \@addtoreset{equation}{section} \makeatother
\renewcommand{\theequation}{\thesection.\arabic{equation}}

\newcounter{theorem}

\makeatletter \@addtoreset{theorem}{section} \makeatother

\newtheorem{thm}[theorem]{Theorem}
\newtheorem*{thm*}{Theorem}
\newtheorem{lem}[theorem]{Lemma}
\newtheorem{cor}[theorem]{Corollary}
\newtheorem{prop}[theorem]{Proposition}

\newtheorem{thmX}{Theorem}

\theoremstyle{definition}
\newtheorem{defn}[theorem]{Definition}
\newtheorem{rem}[theorem]{Remark}
\newtheorem{exam}[theorem]{Example}
\newtheorem{constr}[theorem]{Construction}

\newtheorem*{exam*}{Example}

\usepackage[utf8]{inputenc}

\usepackage{xspace}

\usepackage{mathtools}

\newcommand{\changelocaltocdepth}[1]{%
  \addtocontents{toc}{\protect\setcounter{tocdepth}{#1}}%
  \setcounter{tocdepth}{#1}}

\newcommand{\nc}{\newcommand}
\nc{\renc}{\renewcommand}
\nc{\ssec}{\subsection}
\nc{\sssec}{\subsubsection}
\nc{\on}{\operatorname}
\nc{\term}[1]{#1\xspace}

\usepackage{tikz}
\usetikzlibrary{matrix}
\usepackage{tikz-cd}

\tikzset{
  commutative diagrams/.cd,
  arrow style=tikz,
  diagrams={>=latex}}
\makeatletter
\tikzset{
  column sep/.code=\def\pgfmatrixcolumnsep{\pgf@matrix@xscale*(#1)},
  row sep/.code   =\def\pgfmatrixrowsep{\pgf@matrix@yscale*(#1)},
  matrix xscale/.code=%
    \pgfmathsetmacro\pgf@matrix@xscale{\pgf@matrix@xscale*(#1)},
  matrix yscale/.code=%
    \pgfmathsetmacro\pgf@matrix@yscale{\pgf@matrix@yscale*(#1)},
  matrix scale/.style={/tikz/matrix xscale={#1},/tikz/matrix yscale={#1}}}
\def\pgf@matrix@xscale{1}
\def\pgf@matrix@yscale{1}
\makeatother

\usepackage{enumitem}
\setlist[enumerate,1]{label={(\alph*)},itemsep=\parskip,leftmargin=0pt}
\newlist{thmlist}{enumerate}{1}
\setlist[thmlist,1]{label={\em(\roman*)},ref={\upshape{(\roman*)}},itemsep=\parskip,leftmargin=0pt}     
\newlist{defnlist}{enumerate}{1}
\setlist[defnlist,1]{label={(\roman*)},itemsep=\parskip,leftmargin=0pt}
\newlist{inlinelist}{enumerate*}{1}
\setlist[inlinelist,1]{label={(\alph*)}}

\nc{\sA}{\ensuremath{\mathcal{A}}\xspace}
\nc{\sB}{\ensuremath{\mathcal{B}}\xspace}
\nc{\sC}{\ensuremath{\mathcal{C}}\xspace}
\nc{\sD}{\ensuremath{\mathcal{D}}\xspace}
\nc{\sE}{\ensuremath{\mathcal{E}}\xspace}
\nc{\sF}{\ensuremath{\mathcal{F}}\xspace}
\nc{\sG}{\ensuremath{\mathcal{G}}\xspace}
\nc{\sH}{\ensuremath{\mathcal{H}}\xspace}
\nc{\sI}{\ensuremath{\mathcal{I}}\xspace}
\nc{\sJ}{\ensuremath{\mathcal{J}}\xspace}
\nc{\sK}{\ensuremath{\mathcal{K}}\xspace}
\nc{\sL}{\ensuremath{\mathcal{L}}\xspace}
\nc{\sM}{\ensuremath{\mathcal{M}}\xspace}
\nc{\sN}{\ensuremath{\mathcal{N}}\xspace}
\nc{\sO}{\ensuremath{\mathcal{O}}\xspace}
\nc{\sP}{\ensuremath{\mathcal{P}}\xspace}
\nc{\sQ}{\ensuremath{\mathcal{Q}}\xspace}
\nc{\sR}{\ensuremath{\mathcal{R}}\xspace}
\nc{\sS}{\ensuremath{\mathcal{S}}\xspace}
\nc{\sT}{\ensuremath{\mathcal{T}}\xspace}
\nc{\sU}{\ensuremath{\mathcal{U}}\xspace}
\nc{\sV}{\ensuremath{\mathcal{V}}\xspace}
\nc{\sW}{\ensuremath{\mathcal{W}}\xspace}
\nc{\sX}{\ensuremath{\mathcal{X}}\xspace}
\nc{\sY}{\ensuremath{\mathcal{Y}}\xspace}
\nc{\sZ}{\ensuremath{\mathcal{Z}}\xspace}

\nc{\bA}{\ensuremath{\mathbf{A}}\xspace}
\nc{\bB}{\ensuremath{\mathbf{B}}\xspace}
\nc{\bC}{\ensuremath{\mathbf{C}}\xspace}
\nc{\bD}{\ensuremath{\mathbf{D}}\xspace}
\nc{\bE}{\ensuremath{\mathbf{E}}\xspace}
\nc{\bF}{\ensuremath{\mathbf{F}}\xspace}
\nc{\bG}{\ensuremath{\mathbf{G}}\xspace}
\nc{\bH}{\ensuremath{\mathbf{H}}\xspace}
\nc{\bI}{\ensuremath{\mathbf{I}}\xspace}
\nc{\bJ}{\ensuremath{\mathbf{J}}\xspace}
\nc{\bK}{\ensuremath{\mathbf{K}}\xspace}
\nc{\bL}{\ensuremath{\mathbf{L}}\xspace}
\nc{\bM}{\ensuremath{\mathbf{M}}\xspace}
\nc{\bN}{\ensuremath{\mathbf{N}}\xspace}
\nc{\bO}{\ensuremath{\mathbf{O}}\xspace}
\nc{\bP}{\ensuremath{\mathbf{P}}\xspace}
\nc{\bQ}{\ensuremath{\mathbf{Q}}\xspace}
\nc{\bR}{\ensuremath{\mathbf{R}}\xspace}
\nc{\bS}{\ensuremath{\mathbf{S}}\xspace}
\nc{\bT}{\ensuremath{\mathbf{T}}\xspace}
\nc{\bU}{\ensuremath{\mathbf{U}}\xspace}
\nc{\bV}{\ensuremath{\mathbf{V}}\xspace}
\nc{\bW}{\ensuremath{\mathbf{W}}\xspace}
\nc{\bX}{\ensuremath{\mathbf{X}}\xspace}
\nc{\bY}{\ensuremath{\mathbf{Y}}\xspace}
\nc{\bZ}{\ensuremath{\mathbf{Z}}\xspace}

\nc{\bbA}{\ensuremath{\mathbb{A}}\xspace}
\nc{\bbB}{\ensuremath{\mathbb{B}}\xspace}
\nc{\bbC}{\ensuremath{\mathbb{C}}\xspace}
\nc{\bbD}{\ensuremath{\mathbb{D}}\xspace}
\nc{\bbE}{\ensuremath{\mathbb{E}}\xspace}
\nc{\bbF}{\ensuremath{\mathbb{F}}\xspace}
\nc{\bbG}{\ensuremath{\mathbb{G}}\xspace}
\nc{\bbH}{\ensuremath{\mathbb{H}}\xspace}
\nc{\bbI}{\ensuremath{\mathbb{I}}\xspace}
\nc{\bbJ}{\ensuremath{\mathbb{J}}\xspace}
\nc{\bbK}{\ensuremath{\mathbb{K}}\xspace}
\nc{\bbL}{\ensuremath{\mathbb{L}}\xspace}
\nc{\bbM}{\ensuremath{\mathbb{M}}\xspace}
\nc{\bbN}{\ensuremath{\mathbb{N}}\xspace}
\nc{\bbO}{\ensuremath{\mathbb{O}}\xspace}
\nc{\bbP}{\ensuremath{\mathbb{P}}\xspace}
\nc{\bbQ}{\ensuremath{\mathbb{Q}}\xspace}
\nc{\bbR}{\ensuremath{\mathbb{R}}\xspace}
\nc{\bbS}{\ensuremath{\mathbb{S}}\xspace}
\nc{\bbT}{\ensuremath{\mathbb{T}}\xspace}
\nc{\bbU}{\ensuremath{\mathbb{U}}\xspace}
\nc{\bbV}{\ensuremath{\mathbb{V}}\xspace}
\nc{\bbW}{\ensuremath{\mathbb{W}}\xspace}
\nc{\bbX}{\ensuremath{\mathbb{X}}\xspace}
\nc{\bbY}{\ensuremath{\mathbb{Y}}\xspace}
\nc{\bbZ}{\ensuremath{\mathbb{Z}}\xspace}

\DeclareMathSymbol{A}{\mathalpha}{operators}{`A}
\DeclareMathSymbol{B}{\mathalpha}{operators}{`B}
\DeclareMathSymbol{C}{\mathalpha}{operators}{`C}
\DeclareMathSymbol{D}{\mathalpha}{operators}{`D}
\DeclareMathSymbol{E}{\mathalpha}{operators}{`E}
\DeclareMathSymbol{F}{\mathalpha}{operators}{`F}
\DeclareMathSymbol{G}{\mathalpha}{operators}{`G}
\DeclareMathSymbol{H}{\mathalpha}{operators}{`H}
\DeclareMathSymbol{I}{\mathalpha}{operators}{`I}
\DeclareMathSymbol{J}{\mathalpha}{operators}{`J}
\DeclareMathSymbol{K}{\mathalpha}{operators}{`K}
\DeclareMathSymbol{L}{\mathalpha}{operators}{`L}
\DeclareMathSymbol{M}{\mathalpha}{operators}{`M}
\DeclareMathSymbol{N}{\mathalpha}{operators}{`N}
\DeclareMathSymbol{O}{\mathalpha}{operators}{`O}
\DeclareMathSymbol{P}{\mathalpha}{operators}{`P}
\DeclareMathSymbol{Q}{\mathalpha}{operators}{`Q}
\DeclareMathSymbol{R}{\mathalpha}{operators}{`R}
\DeclareMathSymbol{S}{\mathalpha}{operators}{`S}
\DeclareMathSymbol{T}{\mathalpha}{operators}{`T}
\DeclareMathSymbol{U}{\mathalpha}{operators}{`U}
\DeclareMathSymbol{V}{\mathalpha}{operators}{`V}
\DeclareMathSymbol{W}{\mathalpha}{operators}{`W}
\DeclareMathSymbol{X}{\mathalpha}{operators}{`X}
\DeclareMathSymbol{Y}{\mathalpha}{operators}{`Y}
\DeclareMathSymbol{Z}{\mathalpha}{operators}{`Z}

\nc{\mrm}[1]{\ensuremath{\mathrm{#1}}\xspace}
\nc{\mit}[1]{\ensuremath{\mathit{#1}}\xspace}
\nc{\mbf}[1]{\ensuremath{\mathbf{#1}}\xspace}
\nc{\mcal}[1]{\ensuremath{\mathcal{#1}}\xspace}
\nc{\msc}[1]{\ensuremath{\mathscr{#1}}\xspace}

\renc{\bar}[1]{\overline{#1}}
\DeclarePairedDelimiter\abs{\lvert}{\rvert}%
\nc{\sub}{\subset}
\nc{\too}{\longrightarrow}
\nc{\hook}{\hookrightarrow}
\nc*{\hooklongrightarrow}{\ensuremath{\lhook\joinrel\relbar\joinrel\rightarrow}}
\nc{\hooklong}{\hooklongrightarrow}
\nc{\twoheadlongrightarrow}{\relbar\joinrel\twoheadrightarrow}
\nc{\shiso}{\approx}
\nc{\isoto}{\xrightarrow{\sim}}
\nc{\isofrom}{\xleftarrow{\sim}}

\renc{\ge}{\geqslant}
\renc{\le}{\leqslant}

\nc{\id}{\mathrm{id}}

\DeclareMathOperator{\rk}{\mathrm{rk}}
\DeclareMathOperator{\Hom}{\on{Hom}}
\nc{\uHom}{\underline{\smash{\Hom}}}
\DeclareMathOperator{\Maps}{\on{Maps}}

\DeclareMathOperator{\End}{\on{End}}
\DeclareMathOperator{\Sym}{\on{Sym}}
\nc{\uEnd}{\underline{\smash{\End}}}

\nc{\colim}{\varinjlim}
\renc{\lim}{\varprojlim}
\nc{\Cofib}{\on{Cofib}}
\nc{\Fib}{\on{Fib}}
\nc{\initial}{\varnothing}
\nc{\op}{\mathrm{op}}

\DeclareMathOperator*{\fibprod}{\times}

\renc{\setminus}{\smallsetminus}

\newcommand{\thmref}[1]{Theorem~\ref{#1}}

\newcommand{\secref}[1]{Sect.~\ref{#1}}
\newcommand{\ssecref}[1]{Subsect. ~\ref{#1}}
\newcommand{\sssecref}[1]{\ref{#1}}
\newcommand{\lemref}[1]{Lemma~\ref{#1}}
\newcommand{\propref}[1]{Proposition~\ref{#1}}
\newcommand{\corref}[1]{Corollary~\ref{#1}}

\newcommand{\remref}[1]{Remark~\ref{#1}}
\newcommand{\defnref}[1]{Definition~\ref{#1}}

\renewcommand{\eqref}[1]{(\ref{#1})}

\newcommand{\examref}[1]{Example~\ref{#1}}
\newcommand{\itemref}[1]{\ref{#1}}

\nc{\Spc}{\mrm{Spc}}
\nc{\SCRing}{\mrm{SCRing}}
\nc{\Mod}{\mrm{Mod}}
\nc{\perf}{{\mrm{perf}}}
\nc{\Spt}{\mrm{Spt}}
\nc{\A}{\bA}
\nc{\red}{{\mrm{red}}}
\nc{\Spec}{\on{Spec}}
\nc{\cn}{{\on{cn}}}
\nc{\Sm}{\mrm{Sm}}
\nc{\Sch}{\mrm{Sch}}
\nc{\aff}{{\mrm{aff}}}
\nc{\DSch}{\mrm{DSch}}
\nc{\Qcoh}{\on{Qcoh}}
\nc{\bDelta}{\mathbf{\Delta}}
\nc{\Cech}{\textnormal{\v{C}}}
\nc{\Perf}{\on{Perf}}
\nc{\cl}{{\mrm{cl}}}
\nc{\Tot}{\on{Tot}}
\nc{\idem}{{\mrm{idem}}}
\nc{\K}{\on{K}}
\renc{\P}{\bP}
\nc{\modmod}{/\!\!/}
\nc{\Bl}{\on{Bl}}
\nc{\KH}{\on{KH}}
\nc{\cdh}{\mrm{cdh}}
\nc{\DAlg}{\mrm{DAlg}}
\nc{\Alg}{\mrm{Alg}}

\nc{\scr}{\term{simplicial commutative ring}}
\nc{\scrs}{\term{simplicial commutative rings}}

\nc{\inftyCat}{\term{$\infty$-category}}
\nc{\inftyCats}{\term{$\infty$-categories}}

\nc{\inftyGrpd}{\term{$\infty$-groupoid}}
\nc{\inftyGrpds}{\term{$\infty$-groupoids}}

\nc{\das}{\term{derived algebraic space}}
\nc{\dass}{\term{derived algebraic spaces}}

\nc{\as}{\term{algebraic space}}
\nc{\ass}{\term{algebraic spaces}}

\title{Algebraic K-theory of quasi-smooth blow-ups and cdh descent}
\author{Adeel A. Khan}
\date{2020-02-01}

\makeatletter
\def\l@subsection{\@tocline{2}{0pt}{4pc}{6pc}{}}
\makeatother

\begin{document}

\begin{abstract}
We construct a semi-orthogonal decomposition on the category of perfect complexes on the blow-up of a derived Artin stack in a quasi-smooth centre.
This gives a generalization of Thomason's blow-up formula in algebraic K-theory to derived stacks.
We also provide a new criterion for descent in Voevodsky's cdh topology, which we use to give a direct proof of Cisinski's theorem that Weibel's homotopy invariant K-theory satisfies cdh descent.
\end{abstract}

\thanks{Author partially supported by SFB 1085 Higher Invariants, Universität Regensburg}

\maketitle

\renewcommand\contentsname{\vspace{-1cm}}
\tableofcontents

\parskip 0.2cm
\thispagestyle{empty}

\changelocaltocdepth{1}

\section{Introduction}
\label{sec:intro}

\ssec{}

Let $X$ be a scheme and $i : Z \to X$ a regular closed immersion.
This means that $Z$ is, Zariski-locally on $X$, the zero-locus of some regular sequence of functions $f_1,\ldots,f_n \in \Gamma(X,\sO_X)$.
Then the blow-up $\Bl_{Z/X}$ fits into a square
  \begin{equation}\label{eq:blow-up square}
    \begin{tikzcd}
      \P(\sN_{Z/X}) \ar{r}{i_D}\ar{d}{q}
        & \Bl_{Z/X} \ar{d}{p}
      \\
      Z \ar{r}{i}
        & X,
    \end{tikzcd}
  \end{equation}
where the exceptional divisor is the projective bundle associated to the conormal sheaf $\sN_{Z/X}$, which under the assumptions is locally free of rank $n$.
A result of Thomason \cite{ThomasonBlowup} asserts that after taking algebraic K-theory, the induced square of spectra
  \begin{equation*}
     \begin{tikzcd}
      \K(X) \ar{r}{i^*}\ar{d}{p^*}
        & \K(Z) \ar{d}
      \\
      \K(\Bl_{Z/X}) \ar{r}
        & \K(\P(\sN_{Z/X}))
    \end{tikzcd}
  \end{equation*}
is homotopy cartesian.
Here $\K(X)$ denotes the Bass--Thomason--Trobaugh algebraic K-theory spectrum of perfect complexes on a scheme $X$.
We may summarize this property by saying that algebraic K-theory satisfies \emph{descent} with respect to blow-ups in regularly immersed centres.

Now suppose that $i$ is more generally a \emph{quasi-smooth} closed immersion of derived schemes.
This means that $Z$ is, Zariski-locally on $X$, the \emph{derived} zero-locus of some arbitrary sequence of functions $f_1,\ldots,f_n \in \Gamma(X,\sO_X)$.
(When $X$ is a classical scheme and the sequence is regular, this is the same as the classical zero-locus, and we are in the situation discussed above.)
In the derived setting there is still a conormal sheaf $\sN_{Z/X}$ on $Z$, locally free of rank $n$, and one may still form the blow-up square \eqref{eq:blow-up square}, see \cite{KhanBlowup}.
Our goal in this paper is to generalize Thomason's result above to this situation.
At the same time we also allow $X$ to be a derived \emph{Artin stack}, and consider any \emph{additive invariant} of stable \inftyCats (see \defnref{defn:additive invariant}).
Examples of additive invariants include algebraic K-theory $\K$, connective algebraic K-theory $\K^\mrm{cn}$, topological Hochschild homology $\mrm{THH}$, and topological cyclic homology $\mrm{TC}$.

\begin{thmX}\label{thm:blow-up descent}
Let $E$ be an additive invariant of stable \inftyCats.
Then $E$ satisfies descent by quasi-smooth blow-ups.
That is, given a derived Artin stack $X$ and a quasi-smooth closed immersion $i : Z \to X$ of virtual codimension $n\ge 1$, form the blow-up square \eqref{eq:blow-up square}.
Then the induced commutative square
  \begin{equation*}
    \begin{tikzcd}
      E(X) \ar{r}{i^*}\ar{d}{p^*}
        & E(Z) \ar{d}
      \\
      E(\Bl_{Z/X}) \ar{r}
        & E(\P(\sN_{Z/X}))
    \end{tikzcd}
  \end{equation*}
is homotopy cartesian.
\end{thmX}

We deduce \thmref{thm:blow-up descent} from an analysis of the categories of perfect complexes on $\Bl_{Z/X}$ and on the exceptional divisor $\P(\sN_{Z/X})$.
The relevant notion is that of a \emph{semi-orthogonal decomposition}, see \defnref{defn:SOD}.

\begin{thmX}\label{thm:Perf(P(E))}
Let $X$ be a derived Artin stack.
For any locally free $\sO_X$-module $\sE$ of rank $n+1$, $n\ge 0$, consider the projective bundle $q : \P(\sE) \to X$.
Then we have:
\begin{thmlist}
\item\label{item:Perf(P(E))/fully faithful}
For each $0\le k\le n$, the assignment $\sF \mapsto q^*(\sF) \otimes \sO(-k)$ defines a fully faithful functor $\Perf(X) \to \Perf(\P(\sE))$, whose essential image we denote $\bA(-k)$.

\item\label{item:Perf(P(E))/SOD}
The sequence of full subcategories $(\bA(0),\ldots,\bA(-n))$ forms a semi-orthogonal decomposition of $\Perf(\P(\sE))$.
\end{thmlist}
\end{thmX}

\begin{thmX}\label{thm:Perf(Bl)}
Let $X$ be a derived Artin stack.
For any quasi-smooth closed immersion $i : Z \to X$ of virtual codimension $n\ge 1$, form the blow-up square \eqref{eq:blow-up square}.
Then we have:
\begin{thmlist}
\item\label{item:Perf(Bl)/fully faithful 1}
The assignment $\sF \mapsto p^*(\sF)$ defines a fully faithful functor $\Perf(X) \to \Perf(\Bl_{Z/X})$, whose essential image we denote $\bB(0)$.

\item\label{item:Perf(Bl)/fully faithful 2}
For each $1\le k\le n-1$, the assignment $\sF \mapsto (i_D)_*(q^*(\sF) \otimes \sO(-k))$ defines a fully faithful functor $\Perf(Z) \to \Perf(\Bl_{Z/X})$, whose essential image we denote $\bB(-k)$.

\item\label{item:Perf(Bl)/SOD}
The sequence of full subcategories $(\bB(0),\ldots,\bB(-n+1))$ forms a \emph{semi-orthogonal decomposition} of $\Perf(\Bl_{Z/X})$.
\end{thmlist}
\end{thmX}

We immediately deduce the projective bundle and blow-up formulas
  \begin{equation*}
    E(\P(\sE)) \simeq \bigoplus_{m=0}^{n} E(X),
    \qquad
    E(\Bl_{Z/X}) \simeq E(X) \oplus \bigoplus_{k=1}^{n-1} E(Z),
  \end{equation*}
for any additive invariant $E$, see Corollaries~\ref{cor:E(P(E))} and \ref{cor:E(Bl_{Z/X})}, from which \thmref{thm:blow-up descent} immediately follows (see \ssecref{ssec:blowup/additive}).

\ssec{}

The results mentioned above admit the following interesting special cases:

\begin{enumerate}
  \item
Suppose that $X$ is a smooth projective variety over the field of complex numbers.
This case of \thmref{thm:Perf(P(E))} was proven by Orlov in \cite{OrlovSOD}.
He also proved \thmref{thm:Perf(Bl)} for any \emph{smooth} subvariety $Z \hook X$.
  \item
More generally suppose that $X$ is a quasi-compact quasi-separated classical scheme.
Then the projective bundle formula (\corref{cor:E(P(E))}) for algebraic K-theory was proven by Thomason \cite{ThomasonTrobaugh,ThomasonProjectiveBundle}.
Similarly suppose that $i : Z \to X$ is a quasi-smooth closed immersion of quasi-compact quasi-separated classical schemes.
Then it is automatically a regular closed immersion, and in this case Thomason also proved \corref{cor:E(Bl_{Z/X})} for algebraic K-theory \cite{ThomasonBlowup}.
In fact, the papers \cite{ThomasonProjectiveBundle} and \cite{ThomasonBlowup} essentially contain under these assumptions proofs of Theorems~\ref{thm:Perf(P(E))} and \ref{thm:Perf(Bl)}, respectively, even if the term ``semi-orthogonal decomposition'' is not used explicitly.
For $\mrm{THH}$ and $\mrm{TC}$, these cases of Corollaries~\ref{cor:E(P(E))} and \ref{cor:E(Bl_{Z/X})} were proven by Blumberg and Mandell \cite{BlumbergMandellTHH}.

  \item 
More generally still, let $X$ and $Z$ be classical Artin stacks.
These cases of Theorems~\ref{thm:Perf(P(E))} and \ref{thm:Perf(Bl)} are proven by by Bergh and Schnürer in \cite{BerghSchnuerer}.
However we note that Corollaries~\ref{cor:E(P(E))} and \ref{cor:E(Bl_{Z/X})} were obtained earlier by Krishna and Ravi in \cite{KrishnaRavi}, and their arguments in fact prove Theorems~\ref{thm:Perf(P(E))} and \ref{thm:Perf(Bl)} for classical Artin stacks.

  \item
Let $X$ be a noetherian affine classical scheme, and let $Z$ be the derived zero-locus of some functions $f_1,\ldots,f_n \in \Gamma(X, \sO_X)$.
Then the canonical morphism $i : Z \to X$ is a quasi-smooth closed immersion.
In this case, \thmref{thm:blow-up descent} for algebraic K-theory was proven by Kerz--Strunk--Tamme \cite{KerzStrunkTamme} (where the blow-up $\Bl_{Z/X}$ was explicitly modelled as the derived fibred product $X \fibprod_{\A^n} \Bl_{\{0\}/\A^n}$), as part of their proof of Weibel's conjecture on negative K-theory.
\end{enumerate}

\ssec{}

Let $\KH$ denote homotopy invariant K-theory.
Recall that this is the $\A^1$-localization of the presheaf $X \mapsto \K(X)$.
That is, it is obtained by forcing the property of $\A^1$-homotopy invariance: for every quasi-compact quasi-separated algebraic space $X$, the map
  \begin{equation*}
    \KH(X) \to \KH(X \times \A^1)
  \end{equation*}
is invertible (see \cite{WeibelKH,CisinskiKH}).
As an application of \thmref{thm:blow-up descent}, we give a new proof of the following theorem of Cisinski \cite{CisinskiKH}:

\begin{thmX}\label{thm:KH cdh descent}
The presheaf of spectra $S \mapsto \KH(S)$ satisfies cdh descent on the site of quasi-compact quasi-separated algebraic spaces.
\end{thmX}

This was first proven by Haesemeyer \cite{Haesemeyer} for schemes over a field of characteristic zero, using resolution of singularities.
Cisinski's proof over general bases (noetherian schemes of finite dimension) relies on Ayoub's proper base change theorem in motivic homotopy theory.
A different proof of \thmref{thm:KH cdh descent} (also in the noetherian setting) was recently given by Kerz--Strunk--Tamme \cite[Thm.~C]{KerzStrunkTamme}, as an application of pro-cdh descent and their resolution of Weibel's conjecture on negative K-theory.
The proof we give here is more direct and uses a new criterion for cdh descent (see \thmref{thm:cdh criterion} for a more precise statement):

\begin{thmX}\label{thm:cdh intro}
Let $\sF$ be a Nisnevich sheaf of spectra on the category of quasi-compact quasi-separated algebraic spaces.
Then $\sF$ satisfies cdh descent if and only if it sends closed squares and quasi-smooth blow-up squares to cartesian squares.
\end{thmX}

\thmref{thm:cdh intro} can be compared to a similar criterion due to Haesemeyer, implicit in \cite{Haesemeyer}, which applies to Nisnevich sheaves of spectra on the category of schemes over a field $k$ of characteristic zero.
It asserts that for such a sheaf, cdh descent is equivalent to descent for finite cdh squares and regularly immersed blow-up squares.
Note that the first condition is stronger than descent for closed squares, while the second is weaker than descent for quasi-smooth blow-up squares: regularly immersed blow-up squares are precisely those quasi-smooth blow-up squares where all schemes appearing are underived.
For invariants of stable \inftyCats, a similar cdh descent criterion was noticed independently by Land and Tamme \cite[Thm.~A.2]{LandTamme}.

\thmref{thm:KH cdh descent} was extended to certain nice Artin stacks recently by Hoyois and Hoyois--Krishna \cite{HoyoisCdh,HoyoisKrishna}.
Our cdh descent criterion also applies in that setting (\remref{rem:variants of cdh criterion}\itemref{item:variants of cdh criterion/stacks}) and gives another potential approach to such results.

\ssec{}

The organization of this paper is as follows.
We begin in \secref{sec:prelim} with some background on derived algebraic geometry and on semi-orthogonal decompositions of stable \inftyCats.

\secref{sec:proj} is dedicated to the proof of \thmref{thm:Perf(P(E))}.
We first show that the semi-orthogonal decomposition exists on the larger stable \inftyCat $\Qcoh(\P(\sE))$ (\thmref{thm:Qcoh(P(E))}).
Then we show that it restricts to $\Perf(\P(\sE))$ (\ssecref{ssec:proof of Perf(P(E))}), and deduce the projective bundle formula (\corref{cor:E(P(E))}) for any additive invariant.

We follow a similar pattern in \secref{sec:blowup} to prove \thmref{thm:Perf(Bl)}.
There is a semi-orthogonal decomposition on $\Qcoh(\Bl_{Z/X})$ (\thmref{thm:Qcoh(Bl)}) which then restricts to $\Perf(\Bl_{Z/X})$ (\ssecref{ssec:proof of Perf(Bl)}).
This gives both the blow-up formula (\corref{cor:E(Bl_{Z/X})}) as well as \thmref{thm:blow-up descent} (\sssecref{sssec:proof of blow-up descent}) for additive invariants.
As input we prove a Grothendieck duality statement for virtual Cartier divisors (\propref{prop:omega_D/X}) that should be of independent interest.

\secref{sec:cdh} contains our results on cdh descent and KH.
We first give the general cdh descent criterion (\thmref{thm:cdh criterion}).
We apply this criterion to $\KH$ to give our proof of \thmref{thm:KH cdh descent} (\sssecref{sssec:proof of KH cdh descent}).

\ssec{}

I would like to thank Marc Hoyois, Charanya Ravi, and David Rydh for helpful discussions and comments on previous revisions.
I am especially grateful to David Rydh for pointing out the relevance of the resolution property in \secref{sec:cdh}.

\changelocaltocdepth{2}


\section{Preliminaries}
\label{sec:prelim}

Throughout the paper we work with the language of \inftyCats as in \cite{HTT,HA-20170918}.

\ssec{Derived algebraic geometry}

This paper is set in the world of derived algebraic geometry, as in \cite{HAG2,SAG-20180204,GaitsgoryRozenblyum}.

\sssec{}

Let $\SCRing$ denote the \inftyCat of \scrs.
A \emph{derived stack} is an étale sheaf of spaces $X : \SCRing \to \Spc$.
If $X$ is corepresentable by a \scr $A$, we write $X = \Spec(A)$ and call $X$ an \emph{affine derived scheme}.
A \emph{derived scheme} is a derived stack $X$ that admits a Zariski atlas by affine derived schemes, i.e., a jointly surjective family $(U_i \to X)_i$ of Zariski open immersions with each $U_i$ an affine derived scheme.
Allowing Nisnevich, étale or smooth atlases, respectively, gives rise to the notions of \emph{derived algebraic space}\footnotemark, \emph{derived Deligne--Mumford stack}, and \emph{derived Artin stack}.
\footnotetext{That this agrees with the classical notion of algebraic space (at least under quasi-compactness and quasi-separatedness hypotheses) follows from \cite[Prop.~5.7.6]{RaynaudGruson}.
That it agrees with Lurie's definition follows from \cite[Ex.~3.7.1.5]{SAG-20180204}.}
The precise definition is slightly more involved, see e.g. \cite[Vol.~I, Sect.~4.1]{GaitsgoryRozenblyum}.

Any derived stack $X$ admits an underlying classical stack which we denote $X_\cl$.
If $X$ is a derived scheme, algebraic space, Deligne--Mumford or Artin stack, then $X_\cl$ is a classical such.
For example, $\Spec(A)_\cl = \Spec(\pi_0(A))$ for a \scr $A$.

\sssec{}

Let $X$ be a derived scheme and let $f_1,\ldots,f_n \in \Gamma(X, \sO_X)$ be functions classifying a morphism $f : X \to \A^n$ to affine space.
The \emph{derived zero-locus} of these functions is given by the derived fibred product
  \begin{equation}\label{eq:quasi-smooth}
    \begin{tikzcd}
      Z \ar{r}\ar{d}
        & X \ar{d}{f}
      \\
      \{0\} \ar{r}
        & \A^n.
    \end{tikzcd}
  \end{equation}
If $X$ is classical, then $Z$ is classical if and only if the sequence $(f_1,\ldots,f_n)$ is regular in the sense of \cite{SGA6}, in which case $Z$ is regularly immersed.
A closed immersion of derived schemes $i : Z \to X$ is called \emph{quasi-smooth} (of virtual codimension $n$) if it is cut out Zariski-locally as the derived zero-locus of $n$ functions on $X$.
Equivalently, this means that $i$ is of finite presentation and its shifted cotangent complex $\sN_{Z/X} := \sL_{Z/X}[-1]$ is locally free (of rank $n$).
A closed immersion of derived Artin stacks is quasi-smooth if it satisfies this condition smooth-locally.

A morphism of derived schemes $f : Y \to X$ is quasi-smooth if it can be factored, Zariski-locally on $Y$, through a quasi-smooth closed immersion $i : Y \to X'$ and a smooth morphism $X' \to X$.
A morphism of derived Artin stacks is quasi-smooth if it satisfies this condition smooth-locally on $Y$.
We refer to \cite{KhanBlowup} for more details on quasi-smoothness.

\sssec{}

Important for us is the following construction from \cite{KhanBlowup}.
Given any quasi-smooth closed immersion $i : Z \to X$ of derived Artin stacks, there is an associated \emph{quasi-smooth blow-up square}:
  \begin{equation}\label{eq:qsbl square}
    \begin{tikzcd}
      D \ar{r}{i_D}\ar{d}{q}
        & \Bl_{Z/X} \ar{d}{p}
      \\
      Z \ar{r}{i}
        & X.
    \end{tikzcd}
  \end{equation}
Here $\Bl_{Z/X}$ is the blow-up of $X$ in $Z$, which is a quasi-smooth proper derived Artin stack over $X$, and $D = \P(\sN_{Z/X})$ is the projectivized normal bundle, which is a smooth proper derived Artin stack over $X$.
This square is universal with the following properties: (a) the morphism $i_D$ is a quasi-smooth closed immersion of virtual codimension $1$, i.e., a virtual effective Cartier divisor; (b) the underlying square of classical Artin stacks is cartesian; and (c) the canonical map $q^*\sN_{Z/X} \to \sN_{D/\Bl_{Z/X}}$ is surjective on $\pi_0$.
When $X$ is a derived scheme (resp. derived algebraic space, derived Deligne--Mumford stack), then so is $\Bl_{Z/X}$.

\sssec{}

Given a derived stack $X$, the stable \inftyCat of quasi-coherent sheaves $\Qcoh(X)$ is the limit
  \begin{equation*}
    \Qcoh(X) = \lim_{\Spec(A) \to X} \Qcoh(\Spec(A))
  \end{equation*}
taken over all morphisms $\Spec(A) \to X$ with $A \in \SCRing$.
Here $\Qcoh(\Spec(A))$ is the stable \inftyCat $\Mod_A$ of $A$-modules\footnotemark~in the sense of Lurie.
\footnotetext{Note that if $A$ is discrete (an ordinary commutative ring), then this is not the abelian category of discrete $A$-modules, but rather the derived \inftyCat of this abelian category as in \cite[Chap.~1]{HA-20170918}.}
Informally speaking, a quasi-coherent sheaf $\sF$ on $X$ is thus a collection of quasi-coherent sheaves $x^*(\sF) \in \Qcoh(\Spec(A))$, for every \scr $A$ and every $A$-point $x : \Spec(A) \to X$, together with a homotopy coherent system of compatibilities.

The full subcategory $\Perf(X) \subset \Qcoh(X)$ is similarly the limit
  \begin{equation*}
    \Perf(X) = \lim_{\Spec(A) \to X} \Perf(\Spec(A)),
  \end{equation*}
where $\Perf(\Spec(A))$ is the stable \inftyCat $\Mod_A^\perf$ of \emph{perfect} $A$-modules.
In other words, $\sF \in \Qcoh(X)$ belongs to $\Perf(X)$ if and only if $x^*(\sF)$ is perfect for every \scr $A$ and every morphism $x : \Spec(A) \to X$.

\sssec{}\label{sssec:Qcoh functoriality}

There is an inverse image functor $f^* : \Qcoh(X) \to \Qcoh(Y)$ for any morphism of derived stacks $f : Y \to X$.
It preserves perfect complexes and induces a functor $f^* : \Perf(X) \to \Perf(Y)$.
Regarded as presheaves of \inftyCats, the assignments $X \mapsto \Qcoh(X)$ and $X \mapsto \Perf(X)$ satisfy \emph{descent} for the fpqc topology (\cite[Cor.~D.6.3.3]{SAG-20180204}, \cite[Thm.~1.3.4]{GaitsgoryRozenblyum}).
This means in particular that given any fpqc covering family $(f_\alpha : X_\alpha \to X)_\alpha$, the family of inverse image functors $f_\alpha^* : \Qcoh(X) \to \Qcoh(X_\alpha)$ is jointly conservative.

If $f : Y \to X$ is quasi-compact and \emph{schematic}, in the sense that its fibre over any affine derived scheme is a derived scheme, then there is a direct image functor $f_*$, right adjoint to $f^*$, which commutes with colimits and satisfies a base change formula against inverse images (\cite[Prop.~2.5.4.5]{SAG-20180204}, \cite[Vol.~1, Chap.~3, Prop.~2.2.2]{GaitsgoryRozenblyum}).
If $f$ is proper, locally of finite presentation, and of finite tor-amplitude, then $f_*$ also preserves perfect complexes \cite[Thm.~6.1.3.2]{SAG-20180204}.

\ssec{Semi-orthogonal decompositions}

The following definitions were originally formulated by \cite{BondalKapranov} in the language of triangulated categories and are standard.

\begin{defn}
Let $\bC$ be a stable \inftyCat and $\bD$ a stable full subcategory.
An object $x \in \bC$ is \emph{left orthogonal}, resp. \emph{right orthogonal}, to $\bD$ if the mapping space $\Maps_\bC(x,d)$, resp. $\Maps_\bC(d,x)$, is contractible for all objects $d\in\bD$.
We let $^\perp\bD \subseteq \bC$ and $\bD^\perp \subseteq \bC$ denote the full subcategories of left orthogonal and right orthogonal objects, respectively.
\end{defn}

\begin{defn}\label{defn:SOD}
Let $\bC$ be a stable \inftyCat and let $\bC(0),\ldots,\bC(-n)$ be full stable subcategories.
Suppose that the following conditions hold:
\begin{defnlist}
  \item For all integers $i>j$, there is an inclusion $\bC(i) \subseteq~^\perp\bC(j)$.
  \item The \inftyCat $\bC$ is generated by the subcategories $\bC(0), \ldots, \bC(-n)$, under finite limits and finite colimits.
\end{defnlist}
Then we say that the sequence $(\bC(0),\ldots,\bC(-n))$ forms a \emph{semi-orthogonal decomposition} of $\bC$.
\end{defn}

Semi-orthogonal decompositions of length $2$ come from \emph{split short exact sequences} of stable \inftyCats, as in \cite{BlumbergGepnerTabuada}.

\begin{defn}\leavevmode

\begin{defnlist}
  \item
A \emph{short exact sequence} of small stable \inftyCats is a diagram
  \begin{equation*}
    \bC' \xrightarrow{i} \bC \xrightarrow{p} \bC'',
  \end{equation*}
where $i$ and $p$ are exact, the composite $p\circ i$ is null-homotopic, $i$ is fully faithful, and $p$ induces an equivalence $(\bC/\bC')^\idem \simeq (\bC'')^\idem$ (where $(-)^\idem$ denotes idempotent completion).
  \item
A short exact sequence of small stable \inftyCats
  \begin{equation*}
    \bC' \xrightarrow{i} \bC \xrightarrow{p} \bC''
  \end{equation*}
is \emph{split} if there exist functors $q : \bC \to \bC'$ and $j : \bC'' \to \bC$, right adjoint to $i$ and $p$, respectively, such that the unit $\id \to q\circ i$ and co-unit $p\circ j \to \id$ are invertible.
\end{defnlist}
\end{defn}

\begin{rem}\label{rem:SOD to split SES}
Let $\bC$ be a small stable \inftyCat, and let $(\bC(0), \bC(-1))$ be a semi-orthogonal decomposition.
Then for any object $x \in \bC$, there exists an exact triangle
  \begin{equation*}
    x(0) \to x \to x(-1),
  \end{equation*}
where $x(0) \in \bC(0)$ and $x(-1) \in \bC(-1)$.
To see this, simply observe that the full subcategory spanned by objects $x$ for which such a triangle exists, is closed under finite limits and colimits, and contains $\bC(0)$ and $\bC(-1)$.
Moreover, the assignments $x \mapsto x(0)$ and $x \mapsto x(-1)$ determine well-defined functors $q : \bC \to \bC(0)$ and $p : \bC \to \bC(-1)$, respectively, which are right and left adjoint, respectively, to the inclusions (see e.g. \cite[Rem.~7.2.0.2]{SAG-20180204}).
It follows from this that any semi-orthogonal decomposition $(\bC(0),\bC(-1))$ induces a split short exact sequence
  \begin{equation*}
    \bC(0) \to \bC \xrightarrow{p} \bC(-1).
  \end{equation*}
\end{rem}

\begin{lem}\label{lem:SOD filtration}
Let $\bC$ be a stable \inftyCat, and let $(\bC(0),\ldots,\bC(-n))$ be a sequence of full stable subcategories forming a semi-orthogonal decomposition of $\bC$.
For each $0 \le m \le n$, let $\bC_{\le -m} \subseteq \bC$ denote the full stable subcategory generated by objects in the union $\bC(-m)\cup\cdots\cup\bC(-n)$, and let $\bC_{\le -n-1} \subseteq \bC$ denote the full subcategory spanned by the zero object.
Then there are split short exact sequences
  \begin{equation*}
    \bC_{\le -m-1} \hook \bC_{\le -m} \to \bC(-m)
  \end{equation*}
for each $0\le m\le n$.
\end{lem}

\begin{proof}
It follows from the definitions that for each $0\le m\le n$, the sequence $(\bC(-m), \bC_{\le -m-1})$ forms a semi-orthogonal decomposition of $\bC$.
Therefore the claim follows from \remref{rem:SOD to split SES}.
\end{proof}

\ssec{Additive and localizing invariants}

The following definition is from \cite{BlumbergGepnerTabuada}, except that we do not require commutativity with filtered colimits.

\begin{defn}\label{defn:additive invariant}
Let $\bA$ be a stable presentable \inftyCat.
Let $E$ be an $\bA$-valued functor from the \inftyCat of small stable \inftyCats and exact functors.

\begin{defnlist}
\item
We say that $E$ is an \emph{additive invariant} if for any split short exact sequence
  \begin{equation*}
    \bC' \xrightarrow{i} \bC \xrightarrow{p} \bC'',
  \end{equation*}
the induced map
  \begin{equation*}
    E(\bC') \oplus E(\bC'') \xrightarrow{(i,j)} E(\bC)
  \end{equation*}
is invertible, where $j$ is a right adjoint to $p$.

\item
We say that $E$ is a \emph{localizing invariant} if for any short exact sequence
  \begin{equation*}
    \bC' \xrightarrow{i} \bC \xrightarrow{p} \bC'',
  \end{equation*}
the induced diagram
  \begin{equation*}
    E(\bC') \to E(\bC) \to E(\bC'')
  \end{equation*}
is an exact triangle.
\end{defnlist}
\end{defn}

\begin{rem}
Any localizing invariant is also additive.
\end{rem}

\begin{lem}\label{lem:additive}
Let $\bC$ be a stable \inftyCat, and let $(\bC(0),\ldots,\bC(-n))$ be a sequence of full stable subcategories forming a semi-orthogonal decomposition of $\bC$.
Then for any additive invariant $E$ there is a canonical isomorphism
  \begin{equation*}
    E(\bC) \simeq \bigoplus_{m=0}^n E(\bC(-m)).
  \end{equation*}
\end{lem}

\begin{proof}
Follows immediately from \lemref{lem:SOD filtration}.
\end{proof}

\section{The projective bundle formula}
\label{sec:proj}

\ssec{Projective bundles}
\label{ssec:proj/proj}

Let $X$ be a derived stack and $\sE$ a locally free $\sO_X$-module of finite rank.
Recall that the \emph{projective bundle} associated to $\sE$ is a derived stack $\P(\sE)$ over $X$ equipped with an invertible sheaf $\sO(1)$ together with a surjection $\sE \to \sO(1)$.
More precisely, for any derived scheme $S$ over $X$, with structural morphism $x : S \to X$, the space of $S$-points of $\P(\sE)$ is the space of pairs $(\sL, u)$, where $\sL$ is a locally free $\sO_S$-module of rank $1$, and $u : x^*(\sE) \to \sL$ is surjective on $\pi_0$.
We recall the standard properties of this construction:

\begin{prop}\leavevmode
\begin{thmlist}
  \item
If $f : X' \to X$ is a morphism of derived stacks, then there is a canonical isomorphism $\P(f^*(\sE)) \to \P(\sE) \fibprod_X X'$ of derived stacks over $X'$.
  \item
The projection $\P(\sE)\to X$ is proper and schematic.
In particular, if $X$ is a derived scheme (resp. \das, derived Deligne--Mumford stack, derived Artin stack), then the same holds for the derived stack $\P(\sE)$.
  \item
The relative cotangent complex $\sL_{\P(\sE)/X}$ is canonically isomorphic to $\sF \otimes \sO(-1)$, where the locally free sheaf $\sF$ is the fibre of the canonical map $\sE \to \sO(1)$.
In particular, the morphism $\P(\sE) \to X$ is smooth of relative dimension equal to $\rk(\sE) - 1$.
\end{thmlist}
\end{prop}

\begin{prop}[Serre]\label{prop:Serre}
Let $X$ be a derived Artin stack, and $\sE$ a locally free sheaf of rank $n+1$, $n\ge 0$.
If $q : \P(\sE) \to X$ denotes the associated projective bundle, then we have canonical isomorphisms
  \begin{equation*}
    q_*(\sO(0)) \simeq \sO_X,
    \qquad
    q_*(\sO(-m)) \simeq 0
    \quad (1\le m\le n)
  \end{equation*}
in $\Qcoh(X)$.
\end{prop}

\begin{proof}
There is a canonical map $\sO_X \to q_*(\sO(0))$, the unit of the adjunction $(q^*,q_*)$, and there is a unique map $0 \to q_*(\sO(-m))$ for each $m$.
To show that these are invertible, we may use fpqc descent and base change to the case where $X$ is affine and $\sE$ is free.
Then this is Serre's computation, as generalized to the derived setting by Lurie \cite[Thm.~5.4.2.6]{SAG-20180204}.
\end{proof}

\ssec{Semi-orthogonal decomposition on \texorpdfstring{$\Qcoh(\P(\sE))$}{Qcoh(P(E))}}
\label{ssec:proj/qcoh}

In this subsection we will show that the stable \inftyCat $\Qcoh(\P(\sE))$ admits a canonical semi-orthogonal decomposition.

\begin{thm}\label{thm:Qcoh(P(E))}
Let $X$ be a derived Artin stack.
Let $\sE$ be a locally free $\sO_X$-module of rank $n+1$, $n\ge 0$, and $q : \P(\sE) \to X$ the associated projective bundle.
Then we have:
\begin{thmlist}
  \item\label{item:Qcoh(P(E))/fully faithful}
For every integer $k\in\bZ$, the assignment $\sF \mapsto q^*(\sF)\otimes\sO(k)$ defines a fully faithful functor $\Qcoh(X) \to \Qcoh(\P(\sE))$.
  \item\label{item:Qcoh(P(E))/SOD}
For every integer $k\in\bZ$, let $\bC(k) \subset \Qcoh(\P(\sE))$ denote the essential image of the functor in (i).
Then the subcategories $\bC(k),\ldots,\bC(k-n)$ form a semi-orthogonal decomposition of $\Qcoh(\P(\sE))$.
\end{thmlist}
\end{thm}

We will need the following facts (see Lemmas ~7.2.2.2 and 5.6.2.2 in \cite{SAG-20180204}):

\begin{lem}\label{lem:O(n+1)}
Let $R$ be a \scr and $X = \Spec(R)$.
Denote by $\P^n_R = \P(\sO^{n+1}_X)$ the $n$-dimensional projective space over $R$.
Then for every integer $m\in\bZ$, there is a canonical isomorphism
  \begin{equation*}
    \colim_{J \subsetneq [n]} \sO(m+\abs{J}) \isoto \sO(m+n+1)
  \end{equation*}
in $\Qcoh(\P^n_R)$, where the colimit is taken over the proper subsets $J$ of the set $[n] = \{0,1,\ldots,n\}$, and $0\le\abs{J}\le n$ denotes the cardinality of such a subset.
\end{lem}

\begin{lem}\label{lem:surjection from O(m)'s}
Let $R$ be a \scr and $X = \Spec(R)$.
Denote by $\P^n_R = \P(\sO^{n+1}_X)$ the $n$-dimensional projective space over $R$.
Then for any connective quasi-coherent sheaf $\sF \in \Qcoh(\P^n_R)$, there exists a map
  \begin{equation*}
    \bigoplus_\alpha \sO(d_\alpha) \to \sF,
  \end{equation*}
with $d_\alpha \in \bZ$, which is surjective on $\pi_0$.
\end{lem}

\begin{proof}[Proof of \thmref{thm:Qcoh(P(E))}]
Since the functors $-\otimes\sO(k)$ are equivalences, it will suffice to take $k=0$ in both claims.
For claim (i) we want to show that the unit map $\sF \to q_*q^*(\sF)$ is invertible for all $\sF \in \Qcoh(X)$.
By fpqc descent and base change (\sssecref{sssec:Qcoh functoriality}), we may reduce to the case where $X = \Spec(R)$ is affine and $\sE = \sO^{n+1}_S$ is free.
Now both functors $q^*$ and $q_*$ are exact and moreover commute with arbitrary colimits (the latter by \sssecref{sssec:Qcoh functoriality} since $q$ is quasi-compact and schematic), and $\Qcoh(X) \simeq \Mod_R$ is generated by $\sO_X$ under colimits and finite limits.
Therefore we may assume $\sF = \sO_X$, in which case the claim holds by \propref{prop:Serre}.

For claim (ii), let us first check the orthogonality condition in \defnref{defn:SOD}.
Thus take $\sF,\sG\in \Qcoh(X)$ and consider the mapping space
  \begin{equation*}
    \Maps(q^*(\sF), q^*(\sG) \otimes \sO(-m)) \simeq \Maps(\sF, q_*(\sO(-m)) \otimes \sG)
  \end{equation*}
for $1\le m\le n$, where the identification results from the projection formula.
Since $q_*(\sO(-m)) \simeq 0$ by \propref{prop:Serre}, this space is contractible.

It now remains to show that every $\sF \in \Qcoh(\P(\sE))$ belongs to the full subcategory $\langle \bC(0),\ldots,\bC(-n) \rangle \subseteq \Qcoh(\P(\sE))$ generated under finite colimits and limits by the subcategories $\bC(0),\ldots,\bC(-n)$.
Set $\sG_{-1} = \sF \otimes \sO(-1)$ and define $\sG_m$, for $m\ge 0$, so that we have exact triangles
  \begin{equation}\label{eq:G_m+1}
    q^*q_*(\sG_{m-1} \otimes \sO(1)) \xrightarrow{\mrm{counit}} \sG_{m-1} \otimes \sO(1) \to \sG_{m}.
  \end{equation}
For each $m\ge -1$, we claim that $\sG_m$ is right orthogonal to each of the subcategories $\bC(0),\bC(1),\ldots,\bC(m)$.
For $m=-1$ the claim is vacuous, so take $m\ge 0$ and assume by induction that it holds for $m-1$.
Since $q^*q_*(\sG_{m-1} \otimes \sO(1))$ is contained in $\bC(0)$, it follows that $\sG_{m}$ is right orthogonal to $\bC(0)$.
To show that $\sG_m$ is right orthogonal to $\bC(i)$, for $1\le i\le m$, it will suffice to show that the left-hand and middle terms of the exact triangle \eqref{eq:G_m+1} are both right orthogonal to $\bC(i)$.
For the left-hand term this follows from the inclusion $\bC(0) \subset \bC(i)^\perp$, demonstrated above.
For the middle term $\sG_{m-1} \otimes \sO(1)$, the claim follows by the induction hypothesis.

Now we claim that $\sG_n$ is zero.
Using fpqc descent again, we may assume that $X = \Spec(R)$ and $\sE = \sO^{\oplus n+1}_X$ is free (since the sequence $(\sG_{-1},\sG_0,\ldots,\sG_n)$ is stable under base change).
Using \lemref{lem:surjection from O(m)'s} we can build a map
  \begin{equation*}
    \varphi : \bigoplus_\alpha \sO(m_\alpha)[k_\alpha] \to \sG_n
  \end{equation*}
which is surjective on all homotopy groups.
From \lemref{lem:O(n+1)} it follows that $\sG_n$ is right orthogonal to all $\bC(i)$, $i\in\bZ$.
Thus $\varphi$ must be null-homotopic, so $\sG_n \simeq 0$ as claimed.
Working backwards, we deduce that $\sG_{n-1} \in \bC(-1)$, ..., $\sG_0 \in \langle \bC(-1),\ldots,\bC(-n)\rangle$, and then finally that $\sF \in \langle \bC(0), \bC(-1),\ldots,\bC(-n)\rangle$ as claimed.
\end{proof}

\ssec{Proof of \thmref{thm:Perf(P(E))}}
\label{ssec:proof of Perf(P(E))}

We now deduce \thmref{thm:Perf(P(E))} from \thmref{thm:Qcoh(P(E))}.
First note that the fully faithful functor $\sF \mapsto q^*(\sF)\otimes\sO(k)$ of \thmref{thm:Qcoh(P(E))}\itemref{item:Qcoh(P(E))/fully faithful} restricts to a fully faithful functor $\Perf(X) \to \Perf(\P(\sE))$, since $q^*$ preserves perfect complexes.
This shows \thmref{thm:Perf(P(E))}\itemref{item:Perf(P(E))/fully faithful}.

For part \itemref{item:Perf(P(E))/SOD} we argue again as in the proof of \thmref{thm:Qcoh(P(E))}.
The point is that if $\sF \in \Qcoh(\P(\sE))$ is perfect, then so is each $\sG_m \in \Qcoh(\P(\sE))$, since $q^*$ and $q_*$ preserve perfect complexes (the latter because $q$ is smooth and proper).

\ssec{Projective bundle formula}

From \thmref{thm:Perf(P(E))} and \lemref{lem:additive} we deduce:

\begin{cor}\label{cor:E(P(E))}
Let $X$ be a derived Artin stack, $\sE$ a locally free $\sO_X$-module of rank $n+1$, $n\ge 0$, and $q : \P(\sE) \to X$ the associated projective bundle.
Then for any additive invariant $E$, there is a canonical isomorphism
  \begin{equation*}
    E(\P(\sE)) \simeq \bigoplus_{k=0}^{n} E(X)
  \end{equation*}
induced by the functors $q^*(-)\otimes \sO(-k) : \Perf(X) \to \Perf(\P(\sE))$.
\end{cor}


\section{The blow-up formula}
\label{sec:blowup}

\ssec{Virtual Cartier divisors}

Recall from \cite{KhanBlowup} that a \emph{virtual (effective) Cartier divisor} on a derived Artin stack $X$ is a quasi-smooth closed immersion $i : D \to X$ of virtual codimension $1$.
For any such $i : D \to X$, there is a canonical exact triangle
  \begin{equation*}
    \sO_X(-D) \to \sO_X \to i_*(\sO_D),
  \end{equation*}
where $\sO_X(-D)$ is a locally free sheaf of rank $1$, equipped with a canonical isomorphism $i^*(\sO_X(-D)) \simeq \sN_{D/X}$ (see 3.2.3 and 3.2.9 in \cite{KhanBlowup}).

\begin{lem}\label{lem:i^*i_*(O_D)}
Let $X$ be a derived Artin stack and $i : D \to X$ a virtual Cartier divisor.
Then there is a canonical isomorphism
  \begin{equation*}
    i^*i_*(\sO_D) \simeq \sO_D \oplus \sN_{D/X}[1].
  \end{equation*}
\end{lem}

\begin{proof}
Applying $i^*$ to the exact triangle above (and rotating), we get the exact triangle
  \begin{equation*}
    \sO_D \to i^*i_*(\sO_D) \to \sN_{D/X}[1].
  \end{equation*}
The map $\sO_D \to i^*i_*(\sO_D)$ is induced by the natural transformation $i^*(\eta) : i^* \to i^*i_*i^*$ (where $\eta$ is the adjunction unit), so by the triangle identities it has a retraction given by the co-unit map $i^*i_*(\sO_D) \to \sO_D$.
In other words, the triangle splits.
\end{proof}

\ssec{Grothendieck duality}

Let $i : Z \to X$ be a quasi-smooth closed immersion of derived Artin stacks.
The functor $i_*$ admits a right adjoint $i^!$, which for formal reasons can be computed by the formula
  \begin{equation*}
    i^!(-) \simeq i^*(-) \otimes \omega_{D/X},
  \end{equation*}
where $\omega_{D/X} := i^!(\sO_X)$ is called the \emph{relative dualizing sheaf}.
See \cite[Cor.~6.4.2.7]{SAG-20180204}.
When $i$ is a virtual Cartier divisor, $\omega_{D/X}$ can be computed as follows:

\begin{prop}[Grothendieck duality]\label{prop:omega_D/X}
Let $X$ be a derived Artin stack.
Then for any virtual Cartier divisor $i : D \to X$, there is a canonical isomorphism
  \begin{equation*}
    \sN_{D/X}^\vee[-1] \isoto \omega_{D/X}
  \end{equation*}
of perfect complexes on $D$.
In particular, there is a canonical identification $i^! \simeq i^*(-) \otimes \sN_{D/X}^\vee[-1]$.
\end{prop}

\begin{proof}
Write $\sL := \sO_X(-D)$ and consider again the exact triangle $\sL \to \sO_X \to i_*(\sO_D)$.
By the projection formula, this can be refined to an exact triangle of natural transformations $\id\otimes\sL \to \id \to i_*i^*$, or, passing to right adjoints, an exact triangle $i_*i^! \to \id \to \id \otimes \sL^\vee$.
In particular we get the exact triangle
  \begin{equation}\label{eq:divisor sequence dual}
    i_*i^!(\sO_X) \to \sO_X \to \sL^\vee.
  \end{equation}
The associated map $\sL^\vee[-1] \to i_*i^!(\sO_X)$ gives by adjunction a canonical morphism
  \begin{equation*}
    \sN_{D/X}^\vee[-1] \simeq i^*(\sL^\vee)[-1] \to i^!(\sO_X),
  \end{equation*}
which we claim is invertible.
By fpqc descent and the fact that $i^!$ commutes with the operation $f^*$, for any morphism $f$ \cite[Prop.~6.4.2.1]{SAG-20180204}, we may assume that $X$ is affine.
In this case the functor $i_*$ is conservative, so it will suffice to show that the canonical map
  \begin{equation*}
    i_*(\sN_{D/X}^\vee)[-1] \to i_*i^!(\sO_X)
  \end{equation*}
is invertible.
Considering again the triangle $\sF \otimes \sL \to \sF \to i_*i^*(\sF)$ above and taking $\sF = \sL^\vee$, we get the exact triangle
  \begin{equation*}
    \sO_X \to \sL^\vee \to i_*i^*(\sL^\vee) \simeq i_*(\sN_{D/X}^\vee),
  \end{equation*}
since $\sL$ is invertible.
Comparing with \eqref{eq:divisor sequence dual} yields the claim.
\end{proof}

\ssec{Semi-orthogonal decomposition on \texorpdfstring{$\Qcoh(\Bl_{Z/X})$}{Qcoh(Bl)}}

In this subsection we prove:

\begin{thm}\label{thm:Qcoh(Bl)}
Let $X$ be a derived Artin stack and $i : Z \to X$ a quasi-smooth closed immersion of virtual codimension $n\ge 1$.
Let $\widetilde{X} = \Bl_{Z/X}$ and consider the quasi-smooth blow-up square \eqref{eq:qsbl square}
  \begin{equation*}
    \begin{tikzcd}
      D \ar{r}{i_D}\ar{d}{q}
        & \widetilde{X} \ar{d}{p}
      \\
      Z \ar{r}{i}
        & X
    \end{tikzcd}
  \end{equation*}
Then we have:
\begin{thmlist}
\item\label{item:Qcoh(Bl)/fully faithful}
The functor $p^* : \Qcoh(X) \to \Qcoh(\widetilde{X})$ is fully faithful.
We denote its essential image by $\bD(0) \subset \Qcoh(\widetilde{X})$.

\item\label{item:Qcoh(Bl)/(i_D)_*q^* fully faithful}
The functor $(i_D)_*(q^*(-) \otimes \sO(-k)) : \Qcoh(Z) \to \Qcoh(\widetilde{X})$ is fully faithful, for each $1\le k\le n-1$.
We denote its essential image by $\bD(-k) \subset \Qcoh(\widetilde{X})$.

\item\label{item:Qcoh(Bl)/orthogonal}
For each $1\le k\le n-1$, the full stable subcategory $\bD(-k) \subset \Qcoh(\widetilde{X})$ is right orthogonal to each of $\bD(0),\ldots,\bD(-k+1)$.

\item\label{item:Qcoh(Bl)/generates}
The stable \inftyCat $\Qcoh(\widetilde{X})$ is generated by the full subcategories $\bD(0)$, $\bD(-1)$, \ldots, $\bD(-n+1)$ under finite colimits and finite limits.
In particular, the sequence $(\bD(0), \bD(-1), \ldots, \bD(-n+1))$ forms a semi-orthogonal decomposition of $\Qcoh(\widetilde{X})$.
\end{thmlist}
\end{thm}

\sssec{Proof of \itemref{item:Qcoh(Bl)/fully faithful}}

The claim is that for any $\sF \in \Qcoh(X)$, the unit map $\sF \to p_*p^*(\sF)$ is invertible.
By fpqc descent we may reduce to the case where $X$ is affine and $i$ fits in a cartesian square of the form \eqref{eq:quasi-smooth}.
Since $\Qcoh(X)$ is then generated under colimits and finite limits by $\sO_X$, and $p_*$ commutes with colimits since $p$ is quasi-compact and schematic (\sssecref{sssec:Qcoh functoriality}), we may assume that $\sF = \sO_X$.
In other words, it suffices to show that the canonical map $\sO_X \to p_*(\sO_{\widetilde{X}})$ is invertible.
  \begin{equation*}
    \begin{tikzcd}
      D \ar{r}{i_D}\ar{d}{q}
        & \widetilde{X}\ar{d}{p}
      \\
      Z \ar{r}{i}
        & X
    \end{tikzcd}
    \qquad
    \begin{tikzcd}
      \P^{n-1} \ar{r}\ar{d}
        & \Bl_{\{0\}/\A^n}\ar{d}{p_0}
      \\
      \{0\} \ar{r}{i_0}
        & \A^n,
    \end{tikzcd}
  \end{equation*}
Since the left-hand square is the (derived) base change of the right-hand square along the morphism $f : X \to \A^n$, it follows that the map $\sO_X \to p_*(\sO_{\widetilde{X}})$ is the inverse image of the canonical map $\sO_{\A^n} \to (p_0)_*(\sO_{\Bl_{\{0\}/\A^n}})$.
Thus we reduce to the case where $i$ is the immersion $\{0\} \to \A^n$.
This is well-known, see \cite[Exp.~VII]{SGA6}.

\sssec{Proof of \itemref{item:Qcoh(Bl)/(i_D)_*q^* fully faithful}}

It suffices to show the unit map $\sF \to q_*(i_D)^!(i_D)_*q^*(\sF)$ is invertible for all $\sF \in \Qcoh(Z)$.
As in the previous claim we may assume $X$ is affine and that $\sF = \sO_Z$.
Using \propref{prop:omega_D/X}, the canonical identification $\sN_{D/\widetilde{X}} \simeq \sO_D(1)$, and \lemref{lem:i^*i_*(O_D)}, the unit map is identified with
  \begin{equation*}
    \sO_Z
      \to q_*((i_D)^*(i_D)_*(\sO_D) \otimes \sO_D(-1))[-1]
      \simeq q_*(\sO_D(-1)) \oplus q_*(\sO_D).
  \end{equation*}
Since $q : D \to Z$ is the projection of the projective bundle $\P(\sN_{Z/X})$, it follows from \propref{prop:Serre} that we have identifications $q_*(\sO_D(-1)) \simeq 0$ and $q_*(\sO_D) \simeq \sO_Z$, under which the map in question is the identity.

\sssec{Proof of \itemref{item:Qcoh(Bl)/orthogonal}}

To see that $\bD(-k)$ is right orthogonal to $\bD(0)$, observe that by \thmref{thm:Qcoh(P(E))}, the mapping space
  \begin{equation*}
    \Maps(p^*(\sF_X), (i_D)_*(q^*(\sF_Z) \otimes \sO(-k)))
      \simeq \Maps(q^*i^*(\sF_X), q^*(\sF_Z) \otimes \sO(-k))
  \end{equation*}
is contractible for every $\sF_X \in \Qcoh(X)$ and $\sF_Z \in \Qcoh(Z)$.

To see that $\bD(-k)$ is right orthogonal to $\bD(-k')$, for $1\le k'< k$, consider the mapping space
  \begin{equation*}
    \Maps((i_D)_*(q^*(\sF_Z) \otimes \sO(-k')), (i_D)_*(q^*(\sF'_Z) \otimes \sO(-k))),
  \end{equation*}
for $\sF_Z,\sF'_Z \in \Qcoh(Z)$.
Using fpqc descent and base change for $(i_D)_*$ against $f^*$ for any morphism $f : U \to \widetilde{X}$, we may reduce to the case where $X$ is affine.
Since $\Qcoh(Z)$ is then generated under colimits and finite limits by $\sO_Z$, we may assume that $\sF_Z = \sF'_Z = \sO_Z$.
Then we have
  \begin{align*}
    \Maps((i_D)_*(\sO(-k')), (i_D)_*(\sO(-k)))
      &\simeq \Maps((i_D)^*(i_D)_*(\sO(-k')), \sO(-k))\\
      &\simeq \Maps(\sO(-k') \oplus \sO(-k'+1)[1], \sO(-k))
  \end{align*}
by \lemref{lem:i^*i_*(O_D)} and the projection formula, and this space is contractible by \thmref{thm:Qcoh(P(E))}.

\sssec{Proof of \itemref{item:Qcoh(Bl)/generates}}

Denote by $\bD$ the full subcategory of $\Qcoh(\widetilde{X})$ generated by $\bD(0)$, $\bD(-1)$, \ldots, $\bD(-n+1)$ under finite colimits and finite limits.
The claim is that the inclusion $\bD \subseteq \Qcoh(\widetilde{X})$ is an equality.
Note that $\sO_{\widetilde{X}} \in \bD(0) \subset \bD$ and $(i_D)_*(\sO_D(-k)) \in \bD(-k) \subset \bD$ for $1\le k\le n-1$.
Consider the exact triangle $\sO_{\widetilde{X}}(-D) \to \sO_{\widetilde{X}} \to (i_D)_*(\sO_D)$ and recall that $\sO_{\widetilde{X}}(-D) \simeq \sO_{\widetilde{X}}(1)$.
Tensoring with $\sO(-k)$ and using the projection formula, we get the exact triangle
  \begin{equation*}
    \sO_{\widetilde{X}}(-k+1) \to \sO_{\widetilde{X}}(-k) \to (i_D)_*(\sO_D(-k))
  \end{equation*}
for each $1\le k\le n-1$.
Taking $k=1$ we deduce that $\sO_{\widetilde{X}}(-1) \in \bD$.
Continuing recursively we find that $\sO_{\widetilde{X}}(-k) \in \bD$ for all $1\le k\le n-1$.

Now let $\sF \in \Qcoh(\widetilde{X})$.
Denote by $\sG_0 \in \Qcoh(\widetilde{X})$ the cofibre of the co-unit map $p^*p_*(\sF) \to \sF$.
Note that $\sG_0$ is right orthogonal to $\bD(0)$.
For $1\le m\le n-1$ define $\sG_m$ recursively by the exact triangles
  \begin{equation*}
    (i_D)_*(q^*q_*((i_D)^!(\sG_{m-1}) \otimes \sO(m)) \otimes \sO(-m))
      \xrightarrow{\mrm{counit}} \sG_{m-1} \to \sG_m.
  \end{equation*}
Just as in the proof of \thmref{thm:Qcoh(P(E))}, a simple induction argument shows that each $\sG_m$ is right orthogonal to all of the subcategories $\bD(0), \ldots, \bD(m-1)$.
We now claim that $\sG_{n-1}$ is zero; it will follow by recursion that $\sF$ belongs to $\bD$, as desired.

Since the objects $\sG_k$ are stable under base change, we may use fpqc descent and base change to assume that $X$ is affine.
Moreover we may assume that $i: Z\to X$ fits in a cartesian square of the form \eqref{eq:quasi-smooth}.
By \cite[3.3.6]{KhanBlowup}, $p : \widetilde{X} \to X$ factors through a quasi-smooth closed immersion $i' : \widetilde{X} \to \P^{n-1}_X$.
Recall from \lemref{lem:O(n+1)} that there is a canonical isomorphism $\colim_{J \subsetneq [n-1]} \sO(\abs{J}) \simeq \sO(n)$ in $\Qcoh(\P^{n-1}_X)$.
Applying $(i')^*$, we get $\colim_{J \subsetneq [n-1]} \sO_{\widetilde{X}}(\abs{J}) \simeq \sO_{\widetilde{X}}(n)$ in $\Qcoh(\widetilde{X})$.
In particular, every $\sO_{\widetilde{X}}(k)$ belongs to $\bD$ for all $k\in\bZ$.
Recall also that we may find a map $\bigoplus_\alpha \sO(d_\alpha)[n_\alpha] \to i'_*(\sG_{n-1})$ which is surjective on all homotopy groups (\lemref{lem:surjection from O(m)'s}).
By adjunction this corresponds to a map $\bigoplus_\alpha \sO(d_\alpha)[n_\alpha] \to \sG_{n-1}$ (which is also surjective on homotopy groups).
But the source belongs to $\bD$, and the target is right orthogonal to $\bD$, so this map is null-homotopic.
Thus $\sG_{n-1}$ is zero.

\ssec{Proof of \thmref{thm:Perf(Bl)}}
\label{ssec:proof of Perf(Bl)}

We now deduce \thmref{thm:Perf(Bl)} from \thmref{thm:Qcoh(Bl)}.
First note that the fully faithful functor $\sF \mapsto p^*(\sF)$ of \thmref{thm:Qcoh(Bl)}\itemref{item:Qcoh(P(E))/fully faithful} preserves perfect complexes and therefore restricts to a fully faithful functor $\Perf(X) \to \Perf(\Bl_{Z/X})$.
This shows \thmref{thm:Perf(Bl)}\itemref{item:Perf(Bl)/fully faithful 1}.

Similarly, part \itemref{item:Perf(Bl)/fully faithful 2} follows from the fact that the functors $q^*$ and $(i_D)_*$ preserve perfect complexes.
For the latter, this is because $i_D$ is quasi-smooth (and hence of finite presentation and of finite tor-amplitude).

For part \itemref{item:Perf(Bl)/SOD} we argue again as in the proof of \thmref{thm:Qcoh(Bl)}\itemref{item:Qcoh(Bl)/generates}.
The point is that if $\sF \in \Qcoh(\Bl_{Z/X})$ is perfect, then so is each $\sG_m \in \Qcoh(\P(\sE))$, since $q^*$, $q_*$, $(i_D)_*$ and $(i_D)^!$ all preserve perfect complexes.
For the latter this follows from \propref{prop:omega_D/X}.

\ssec{Blow-up formula}
\label{ssec:blowup/additive}

\sssec{}

By \thmref{thm:Perf(Bl)} and \lemref{lem:additive} we get:

\begin{cor}\label{cor:E(Bl_{Z/X})}
Let $X$ be a derived Artin stack and $i : Z \to X$ a quasi-smooth closed immersion of virtual codimension $n\ge 1$.
Then for any additive invariant $E$, there is a canonical isomorphism
  \begin{equation*}
    E(\Bl_{Z/X}) \simeq E(X) \oplus \bigoplus_{k=1}^{n-1} E(Z).
  \end{equation*}
\end{cor}

\sssec{Proof of \thmref{thm:blow-up descent}}
\label{sssec:proof of blow-up descent}

Combine Corollaries~\ref{cor:E(Bl_{Z/X})} and \ref{cor:E(P(E))} (with $\sE = \sN_{Z/X}$).


\section{The cdh topology}
\label{sec:cdh}

\ssec{The cdh topology}

The following notion was introduced by Voevodsky \cite{VoevodskyCdh} for noetherian schemes:

\begin{defn}\label{defn:cdh square}\leavevmode
Suppose given a cartesian square $Q$ of \ass
  \begin{equation}\label{eq:cdh square}
    \begin{tikzcd}
      B \ar{r}\ar{d}
        & Y \ar{d}{p}
      \\
      A \ar{r}{e}
        & X.
    \end{tikzcd}
  \end{equation}
\begin{defnlist}
  \item\label{item:cdh square/Nis square}
We say that $Q$ is a \emph{Nisnevich square} if $e$ is an open immersion, and $p$ is an étale morphism inducing an isomorphism $(Y\setminus B)_\red \simeq (X\setminus A)_\red$.
  \item\label{item:cdh square/proper cdh square}
We say that $Q$ is a \emph{proper cdh square}, or \emph{abstract blow-up square}, if $e$ is a closed immersion of finite presentation, and $p$ is a proper morphism inducing an isomorphism $(Y\setminus B)_\red \simeq (X\setminus A)_\red$.
  \item\label{item:cdh square/cdh square}
We say that $Q$ is a \emph{cdh square} if it is either a Nisnevich square or a proper cdh square.
\end{defnlist}
\end{defn}

\sssec{}

Given any class of commutative squares of \ass, we say that a presheaf satisfies \emph{descent} for this class if it sends all such squares to homotopy cartesian squares, and the empty scheme to a terminal object.
In case of the three classes considered in \defnref{defn:cdh square}, it follows from a theorem of Voevodsky \cite[Cor.~5.10]{VoevodskyCD} that descent in this sense is equivalent to \v{C}ech descent with respect to the associated Grothendieck topology.

\begin{exam}\label{exam:Nis descent for E localizing}
Every localizing invariant $E$ satisfies Nisnevich descent when regarded as a presheaf on quasi-compact quasi-separated \ass with $E(X) = E(\Perf(X))$.
This is essentially due to Thomason \cite{ThomasonTrobaugh} and in the asserted generality is a consequence of the study of compact generation properties of the \inftyCats $\Qcoh(X)$ carried out by Bondal--Van den Bergh \cite{BondalVDB}.
\end{exam}

\begin{exam}\label{exam:qsbl square cl}
Any quasi-smooth blow-up square \eqref{eq:qsbl square} induces a proper cdh square
  \begin{equation*}
    \begin{tikzcd}
      \P(\sN_{Z/X}|_{Z_\cl}) \ar{r}\ar{d}
        & (\Bl_{Z/X})_\cl \ar{d}
      \\
      Z_\cl \ar{r}
        & X_\cl
    \end{tikzcd}
  \end{equation*}
on underlying classical \ass.
\end{exam}

\begin{exam}\label{exam:closed square}
Consider the class of proper cdh squares \eqref{eq:cdh square} where the proper morphism $p$ is a closed immersion (with quasi-compact open complement).
The associated Grothendieck topology is the same as the one generated by \emph{closed squares}, i.e. cartesian squares as in \eqref{eq:cdh square} such that $e$ and $p$ are closed immersions, $e$ is of finite presentation and $p$ has quasi-compact open complement, and $A \sqcup Y \to X$ is surjective on underlying topological spaces.
\end{exam}

\begin{exam}\label{exam:nil-immersion}
Note that for any \as $X$, the square
  \begin{equation*}
    \begin{tikzcd}
      \initial \ar{r}\ar{d}
        & X_\red \ar{d}
      \\
      \initial \ar{r}
        & X
    \end{tikzcd}
  \end{equation*}
is a closed square as in \examref{exam:closed square}.
\end{exam}

\ssec{A cdh descent criterion}

\begin{thm}\label{thm:cdh criterion}
Let $\sF$ be a presheaf on the category $\bC$ of \ass, with values in a stable \inftyCat.
Then $\sF$ satisfies cdh descent if and only if it satisfies the following conditions:
\begin{thmlist}
  \item\label{item:cdh criterion/reduced}
It sends the empty scheme to a zero object.
  \item\label{item:cdh criterion/Nis}
It sends Nisnevich squares to cartesian squares.
  \item\label{item:cdh criterion/closed}
It sends closed squares to cartesian squares.
  \item\label{item:cdh criterion/qsbl}
For every $X\in\bC$ and every quasi-smooth closed immersion $Z \to X$, it sends the square (\examref{exam:qsbl square cl})
    \begin{equation*}
      \begin{tikzcd}
        \P(\sN_{Z/X}|_{Z_\cl}) \ar{r}\ar{d}
          & (\Bl_{Z/X})_\cl \ar{d}
        \\
        Z_\cl \ar{r}
          & X
      \end{tikzcd}
    \end{equation*}
to a cartesian square.
\end{thmlist}
Moreover, the same holds if $\bC$ is replaced by the full subcategory of \begin{inlinelist}
  \item quasi-compact quasi-separated (qcqs) \ass,
  \item schemes,
  \item or qcqs schemes
\end{inlinelist}.
\end{thm}

\begin{rem}\label{rem:trivial extension to dass}
Any presheaf $\sF$ on \ass can be trivially extended to \dass, by setting $\Gamma(X, \sF) = \Gamma(X_\cl, \sF)$ for every \das $X$.
The condition~\itemref{item:cdh criterion/qsbl} in \thmref{thm:cdh criterion} is equivalent to requiring this extension to satisfy descent for quasi-smooth blow-up squares \eqref{eq:qsbl square}.
\end{rem}

\begin{exam}\label{exam:cdh=localizing+nil+closed}
Let $E$ be a localizing invariant of stable \inftyCats.
Then it satisfies Nisnevich descent on qcqs \ass (\examref{exam:Nis descent for E localizing}) and quasi-smooth blow-up descent (\thmref{thm:blow-up descent}).
Assume that $E$ also satisfies \emph{derived nilpotent invariance}, i.e., that the canonical map $E(X) \to E(X_\cl)$ is invertible for every \das $X$.
Then the condition \itemref{item:cdh criterion/qsbl} in \thmref{thm:cdh criterion} holds.
Therefore, $E$ satisfies cdh descent if and only if it satisfies closed descent.
Moreover, by Nisnevich descent it suffices to consider closed squares of affine schemes.
\end{exam}

\begin{exam}\label{exam:A1 closed}
In the presence of $\A^1$-homotopy invariance, the Morel--Voevodsky localization theorem \cite[Theorem~3.2.21]{MorelVoevodsky} provides the following sufficient condition for closed descent.
Let $\sF$ be an $\A^1$-invariant Nisnevich sheaf on the category of \ass.
Suppose that, for every \as $S$, its restriction $\sF_S$ to the site of \emph{smooth} \ass over $S$ is stable under arbitrary base change.
That is, for every morphism of \ass $f : T \to S$, the canonical map $f^*(\sF_S) \to \sF_T$ is invertible.
Then $\sF$ satisfies closed descent.
This follows immediately from the closed base change formula (cf. \cite[Prop.~3.3.2]{KhanLocalization}).
\end{exam}

\begin{rem}
Let $E$ be a localizing invariant and suppose that it is moreover \emph{truncating} in the sense of \cite{LandTamme}.
That is, if $R$ is a connective $\sE_1$-ring spectrum and $\Mod_R^\perf$ denotes the stable \inftyCat of left $R$-modules, then the canonical map $E(\Mod^\perf_R) \to E(\Mod^\perf_{\pi_0(R)})$ is invertible.
Then Land--Tamme have recently proven that $E$ has closed descent, at least if we restrict to \emph{noetherian} \ass (see Step~1 in the proof of \cite[Thm.~A.2]{LandTamme}).
\end{rem}

\begin{rem}\label{rem:variants of cdh criterion}
There are a few variants of \thmref{thm:cdh criterion} with the same proof.
For example:
\begin{defnlist}
  \item
On the category of (qcqs) schemes, descent with respect to the \emph{rh} topology (generated by Zariski squares and proper cdh squares) can be checked with the same criteria, except that Nisnevich squares are replaced by Zariski squares in condition \itemref{item:cdh criterion/Nis}.
  \item
If we do not assume either Nisnevich or Zariski descent, descent for the proper cdh topology is still equivalent to conditions \itemref{item:cdh criterion/reduced}, \itemref{item:cdh criterion/closed}, and \itemref{item:cdh criterion/qsbl}, as long as we restrict to a full subcategory of algebraic spaces or schemes which satisfy Thomason's \emph{resolution property}.
For example, this holds on the category of quasi-projective schemes.
  \item\label{item:variants of cdh criterion/stacks}
One can extend the criterion to qcqs Artin stacks as follows.
The definition of Nisnevich square extends without modification (cf. \cite[Subsect.~2.3]{HoyoisKrishna}).
In the definition of proper cdh square, we add the requirement that the proper morphism $p$ is representable (cf. \emph{op. cit.}).
Then the criterion of \thmref{thm:cdh criterion} holds for stacks \emph{which admit the resolution property Nisnevich-locally}, see (\sssecref{sssec:cdh stacks}).
This condition is relatively mild.
For example, many quotient stacks have the resolution property (\cite[Lem.~2.4]{ThomasonEquivariant}, \cite[Exam.~7.5]{HallRydhPerfect}).
By the Nisnevich-local structure theorem of Alper--Hall--Rydh \cite[Thm.~2.9]{HoyoisKrishna}, any stack with linearly reductive and almost multiplicative stabilizers satisfies the resolution property Nisnevich-locally.
\end{defnlist}
\end{rem}

\ssec{Proof of \thmref{thm:cdh criterion}}

Since Nisnevich squares, closed squares, and quasi-smooth blow-up squares are all cdh squares, the conditions are clearly necessary.
Conversely suppose that $\sF$ is a presheaf satisfying the conditions and let $Q$ be a proper cdh square of \ass of the form
  \begin{equation}\label{eq:proper cdh in proof}
    \begin{tikzcd}
      E \ar{r}\ar{d}
        & Y \ar{d}{p}
      \\
      Z \ar{r}{i}
        & X.
    \end{tikzcd}
  \end{equation}
It will suffice to show that the induced square $\Gamma(Q, \sF)$ is homotopy cartesian.

\sssec{}\label{sssec:cdh criterion proof/blow-up}

Assume first that $Q$ is a blow-up square, i.e., that $Y = \Bl_{Z/X}$ is the blow-up of $X$ centred in $Z$ (and $E = \P(\sC_{Z/X})$ is the projectivized normal cone).
By Nisnevich descent we may assume that $X$ satisfies the resolution property (e.g. $X$ is affine).
Since $i : Z \to X$ is of finite presentation, the ideal of definition $\sI \subset \sO_X$ is of finite type.
Thus by the resolution property there exists a surjection $u : \sE \to \sI$ with $\sE$ a locally free $\sO_X$-module of finite rank.
Denote by $V = \bV_X(\sE) = \Spec_X(\Sym_{\sO_X}(\sE))$ the associated vector bundle and $0 : X \to V$ the zero section.
The $\sO_X$-module homomorphism $u : \sE \to \sI \subset \sO_X$ induces a section of $V$, whose derived zero-locus $\widetilde{Z}$ fits in the homotopy cartesian square
  \begin{equation*}
    \begin{tikzcd}
      \widetilde{Z} \ar{r}{\widetilde{i}}\ar{d}
        & X \ar{d}{u}
      \\
      X \ar{r}{0}
        & V.
    \end{tikzcd}
  \end{equation*}
By construction, $\widetilde{i} : \widetilde{Z} \to X$ is a quasi-smooth closed immersion and there is a canonical morphism $Z \to \widetilde{Z}$ which induces an isomorphism $Z \simeq \widetilde{Z}_\cl$.
Regarding $\sF$ as a presheaf on \dass as in \remref{rem:trivial extension to dass}, the square $\Gamma(Q,\sF)$ now factors as follows:
  \begin{equation*}
    \begin{tikzcd}
      \Gamma(X, \sF) \ar{r}\ar{d}
        & \Gamma(\widetilde{Z}, \sF) \ar{d}\ar[equals]{r}
        & \Gamma(Z, \sF) \ar{ldd}
      \\
      \Gamma(\Bl_{\widetilde{Z}/X}, \sF) \ar{r}\ar{d}
        & \Gamma(\P(\sN_{\widetilde{Z}/X}), \sF) \ar{d}
      \\
      \Gamma(\Bl_{Z/X}, \sF) \ar{r}
        & \Gamma(\P(\sC_{Z/X}), \sF)
    \end{tikzcd}
  \end{equation*}
The upper square is induced by a quasi-smooth blow-up square, hence is cartesian.
The lower square is induced by a closed square, hence is also cartesian.
Therefore it follows that the outer composite square is also cartesian.
This shows that $\sF$ satisfies descent for blow-up squares.

\sssec{}\label{sssec:cdh criterion proof/blow-up (X-Z)-admissible}

Slightly more generally, suppose that $Y = \Bl_{Z'/X}$ is a blow-up centred in some closed immersion $Z' \to X$ with $\abs{Z'} \subseteq \abs{Z}$ on underlying topological spaces, and let $E' \to Y$ denote the exceptional divisor.
Since $\sF$ is invariant under nilpotent extensions (\examref{exam:nil-immersion}) we may assume that $i' : Z' \to X$ actually factors through a closed immersion $Z' \to Z$ (see \examref{exam:nil-immersion}).
Applying descent for blow-up squares (\sssecref{sssec:cdh criterion proof/blow-up}), it will suffice to show that $\sF$ satisfies descent for the square
  \begin{equation*}
    \begin{tikzcd}
      E' \ar{r}\ar{d}
        & E \ar{d}
      \\
      Z' \ar{r}
        & Z.
    \end{tikzcd}
  \end{equation*}
Note that the blow-up $\Bl_{Z'/Z}$ is equipped with a canonical closed immersion into $E$ so that $E' \to E$ and $\Bl_{Z'/Z} \to Z$ form a closed covering.
Applying closed descent and descent for the blow-up square associated to $Z' \to Z$ (\sssecref{sssec:cdh criterion proof/blow-up}), we conclude.

\sssec{}

For the case of an arbitrary proper cdh square, we recall the standard argument using Raynaud--Gruson's technique of \emph{platification par éclatements} \cite[I, Cor.~5.7.12]{RaynaudGruson} to reduce to the case considered above (see e.g. \cite[Subsect.~5.2]{KerzStrunkTamme}).

\begin{constr}
  Let $Q$ be a proper cdh square of the form \eqref{eq:proper cdh in proof}.
  Assume that $X$ is quasi-compact and quasi-separated and that the open subspace $X\setminus Z$ is quasi-compact and dense in $X$.
  Then there exists a proper cdh square $Q'$ sitting above $Q$ such that the composite $Q' \circ Q$
    \begin{equation*}
      \begin{tikzcd}
        E' \ar{r}\ar{d}\ar[phantom]{rd}{\scriptstyle Q'}
        & Y' \ar{d}
      \\
      E \ar{r}\ar{d}\ar[phantom]{rd}{\scriptstyle Q}
        & Y \ar{d}{p}
      \\
      Z \ar[swap]{r}{i}
        & X
      \end{tikzcd}
    \end{equation*}
  is of the form considered in (\sssecref{sssec:cdh criterion proof/blow-up (X-Z)-admissible}).
  That is, $Y'$ is a blow-up of $X$ centred in some closed immersion $Z' \to X$ with $\abs{Z'} \subseteq \abs{Z}$.
  This follows from Raynaud--Gruson just as in the proof of \cite[Cor.~2.4]{HoyoisKrishna}.
\end{constr}

Let $Q$ be a proper cdh square of the form \eqref{eq:proper cdh in proof}.
Since $\sF$ is a Nisnevich sheaf, we may assume that $X$ is quasi-compact and quasi-separated.
Since $i : Z \to X$ is of finite presentation, its open complement $U = X\setminus Z$ is quasi-compact.
Using closed descent, we can ensure that $U$ is dense in $X$.
Now apply the construction above to get a proper cdh square $Q'$ such that $Q'\circ Q$ is of the form considered in (\sssecref{sssec:cdh criterion proof/blow-up (X-Z)-admissible}).
Applying the construction again, this time to $Q'$, we end up with a third square $Q''$ such that the composite $Q'' \circ Q'$ is also of the form considered in (\sssecref{sssec:cdh criterion proof/blow-up (X-Z)-admissible}).
Then we know that $\Gamma(Q'\circ Q, \sF)$ and $\Gamma(Q''\circ Q', \sF)$ are both homotopy cartesian.
It follows that the square $\Gamma(Q', \sF)$ is also homotopy cartesian (since $\sF$ takes values in a stable \inftyCat, it suffices to check that the induced map on homotopy fibres is invertible), and hence so is $\Gamma(Q, \sF)$.

\sssec{}
\label{sssec:cdh stacks}

We now discuss the extension to stacks mentioned in \remref{rem:variants of cdh criterion}\itemref{item:variants of cdh criterion/stacks}.
The precise statement is as follows.
Let $\bC$ be a category of qcqs Artin stacks such that
\begin{inlinelist}
  \item
every stack $X\in\bC$ admits a Nisnevich atlas by stacks with the resolution property;
  \item
for every stack $X\in\bC$ and every blow-up $Y \to X$, the qcqs Artin stack $Y$ also belongs to $\bC$.
\end{inlinelist}
Then the statement of \thmref{thm:cdh criterion} holds for presheaves on $\bC$.

The proof for the case of a blow-up square (\sssecref{sssec:cdh criterion proof/blow-up}) has been presented in such a way that it holds \emph{mutatis mutandis} under the above assumptions.
The argument of \cite[Claim~5.3]{KerzStrunkTamme} also goes through, using descent for closed squares and blow-up squares, to deal with the slightly more general case where $Y = \Bl_{Z'/X}$ is a blow-up centred in some closed immersion $Z' \to X$ that factors through $Z$.
To reduce a general proper cdh square to that case, we use Rydh's extension of Raynaud--Gruson \cite[Thm.~2.2]{HoyoisKrishna}.
First, closed descent allows us to assume that $X\setminus Z$ is dense in $X$.
Then we apply Rydh--Raynaud--Gruson just as in the proof of \cite[Cor.~2.4]{HoyoisKrishna}.
The only difference with the case of schemes or algebraic spaces is that in general we get a \emph{sequence} of $(X\setminus Z)$-admissible blow-ups $\tilde{X} \to X$ which factors through $p : Y \to X$.
The addition of a simple induction is then the only modification required to run the same argument.

\ssec{Homotopy invariant K-theory}
\label{ssec:kh/kh}

\sssec{}

For any qcqs \as $X$, its homotopy invariant K-theory spectrum is given by the formula
  \begin{equation}\label{eq:Gamma(X, Lhtp(K))}
    \Gamma(X, \KH) = \colim_{[n] \in \bDelta^\op} \K(X \times \A^n).
  \end{equation}
In other words, $\Gamma(X, \KH)$ is the geometric realization of the simplicial diagram $\K(X \times \A^\bullet)$, where $\A^\bullet$ is regarded as a cosimplicial scheme in the usual way (see e.g. \cite[p.~45]{MorelVoevodsky}).
This extends the usual definition \cite{WeibelKH,ThomasonTrobaugh}, and is a way to formally impose the property of $\A^1$-homotopy invariance: for any qcqs \as $X$, the projection $p : X \times \A^1 \to X$ induces an isomorphism of spectra
  \begin{equation*}
    p^* : \Gamma(X, \KH) \to \Gamma(X \times \A^1, \KH).
  \end{equation*}

\sssec{}
\label{sssec:KH derived}

As the previous paragraph makes sense when $X$ is derived, we may regard $\KH$ as a presheaf on qcqs derived \ass.
Given a Nisnevich square of the form \eqref{eq:cdh square}, Nisnevich descent for K-theory (\examref{exam:Nis descent for E localizing}) yields homotopy cartesian squares of spectra
  \begin{equation*}
    \begin{tikzcd}
      \K(X \times \A^n) \ar{r}\ar{d}
        & \K(A \times \A^n) \ar{d}
      \\
      \K(Y \times \A^n) \ar{r}
        & \K(B \times \A^n)
    \end{tikzcd}
  \end{equation*}
for every $[n] \in \bDelta^\op$.
Passing to the colimit over $n$, we deduce that $\KH$ also satisfies Nisnevich descent.
We have:

\begin{thm}\label{thm:KH nil}
  For every qcqs \das $S$, the canonical morphism of spectra
    \begin{equation*}
      \Gamma(S, \KH) \to \Gamma(S_\cl, \KH)
    \end{equation*}
  is invertible.
\end{thm}

\begin{proof}
By Nisnevich descent, we may as well assume $S$ is an affine derived scheme.
Let $\KH_S$ denote the restriction of $\KH$ to the site of affine derived schemes that are smooth and of finite presentation over $S$.
This is still an $\A^1$-homotopy invariant Nisnevich sheaf, and it is equipped with a canonical morphism
  \begin{equation*}
    \K^\cn_S \to \K_S \to \KH_S,
  \end{equation*}
where $\K^\cn_S$ and $\K_S$ are the respective restrictions of connective and non-connective K-theory.
By Cisinski, this morphism exhibits $\KH_S$ as the \emph{Bott periodization} of the $\A^1$-localization of $\K^\cn_S$, i.e., the periodization with respect to the Bott element $b \in \K_1(\bG_{m,S})$ (the proof is the same as in the case where $S$ is classical \cite[Cor.~2.12]{CisinskiKH}).
It follows from this description that for any morphism of affine derived schemes $f : T \to S$, there is a canonical isomorphism $f^*(\KH_S) \simeq \KH_T$, where $f^*$ denotes the functor of inverse image of $\A^1$-invariant Nisnevich sheaves.
Indeed, we reduce to checking the same property for $\K^\cn_S$, which is clear as this is identified up to Zariski localization with the group completion of the presheaf $\coprod_n \mrm{BGL}_{n,S}$.

In particular, we get a canonical isomorphism $i^*(\KH_S) \simeq \KH_{S_\cl}$, where $i : S_\cl \to S$ is the inclusion of the underlying classical scheme.
Moreover, $i^*$ induces an equivalence between the \inftyCats of $\A^1$-invariant Nisnevich sheaves on $S$ and $S_\cl$, respectively, by \cite[Cor.~1.3.5]{KhanThesis}.
We deduce that the canonical morphism
  \begin{equation*}
    \Gamma(S, \KH)
      \simeq \Gamma(S, \KH_S)
      \to \Gamma(S_\cl, \KH_{S_\cl})
      \simeq \Gamma(S_\cl, \KH)
  \end{equation*}
is invertible.
\end{proof}

\sssec{Proof of \thmref{thm:KH cdh descent}}
\label{sssec:proof of KH cdh descent}

We use the criterion of \thmref{thm:cdh criterion}.
Condition~\itemref{item:cdh criterion/reduced} is obvious.
Nisnevich descent (condition~\itemref{item:cdh criterion/Nis}) was verified above (\sssecref{sssec:KH derived}).
For condition~\itemref{item:cdh criterion/qsbl}, it will suffice by \thmref{thm:KH nil} and \remref{rem:trivial extension to dass} to show that $\KH$ sends quasi-smooth blow-up squares of \dass to homotopy cartesian squares.
This follows from the same property for K-theory (\thmref{thm:blow-up descent}) using the formula \eqref{eq:Gamma(X, Lhtp(K))} (just as in the proof of Nisnevich descent).
For closed descent (condition~\itemref{item:cdh criterion/closed}), we may restrict our attention to closed squares of \emph{affine} schemes (by Nisnevich descent).
This is classical, see \cite[Exer.~9.11(f)]{ThomasonTrobaugh} or \cite[Cor.~4.10]{WeibelKH}.
Alternatively, it follows from the criterion of \examref{exam:A1 closed}.

\begin{rem}
By continuity for $\KH$ (e.g. \cite[Thm.~4.9(5)]{HoyoisKrishna}), once we have descent for proper cdh squares as in \defnref{defn:cdh square}\itemref{item:cdh square/proper cdh square}, we can immediately drop the finite presentation hypothesis on $e$.
\end{rem}

\bibliographystyle{alphamod}
{\small
\bibliography{references}

\begin{thebibliography}{BVdB03}

\bibitem[SGA 6]{SGA6}
Pierre Berthelot, Alexandre Grothendieck, and Luc Illusie.
\newblock {\em S\'eminaire de G\'eom\'etrie Alg\'ebrique du {B}ois {M}arie -
  1966-67 - {T}h\'eorie des intersections et th\'eor\`eme de {R}iemann-{R}och -
  ({S}{G}{A} 6)}, volume 225.
\newblock Springer-Verlag, 1971.

\bibitem[BGT13]{BlumbergGepnerTabuada}
Andrew~J Blumberg, David Gepner, and Gon{\c{c}}alo Tabuada.
\newblock A universal characterization of higher algebraic {K}-theory.
\newblock {\em Geometry \& Topology}, 17(2):733--838, 2013.

\bibitem[BK89]{BondalKapranov}
Aleksei~Igorevich Bondal and Mikhail~Mikhailovich Kapranov.
\newblock Representable functors, {S}erre functors, and mutations.
\newblock {\em Izvestiya Rossiiskoi Akademii Nauk. Seriya Matematicheskaya},
  53(6):1183--1205, 1989.

\bibitem[BM12]{BlumbergMandellTHH}
Andrew~J. Blumberg and Michael~A. Mandell.
\newblock Localization theorems in topological {H}ochschild homology and
  topological cyclic homology.
\newblock {\em Geom. Topol.}, 16(2):1053--1120, 2012.

\bibitem[BS17]{BerghSchnuerer}
Daniel Bergh and Olaf~M Schn{\"u}rer.
\newblock Conservative descent for semi-orthogonal decompositions.
\newblock {\em arXiv preprint arXiv:1712.06845}, 2017.

\bibitem[BVdB03]{BondalVDB}
Aleksei~Igorevich Bondal and Michel Van~den Bergh.
\newblock Generators and representability of functors in commutative and
  noncommutative geometry.
\newblock {\em Moscow Mathematical Journal}, 3(1):1--36, 2003.

\bibitem[Cis13]{CisinskiKH}
Denis-Charles Cisinski.
\newblock Descente par {\'e}clatements en {K}-th{\'e}orie invariante par
  homotopie.
\newblock {\em Annals of Mathematics}, pages 425--448, 2013.

\bibitem[GR17]{GaitsgoryRozenblyum}
Dennis Gaitsgory and Nick Rozenblyum.
\newblock {\em A Study in Derived Algebraic Geometry: Volumes {I} and {II}}.
\newblock American Mathematical Soc., 2017.

\bibitem[Hae04]{Haesemeyer}
Christian Haesemeyer.
\newblock Descent properties of homotopy {K}-theory.
\newblock {\em Duke Mathematical Journal}, 125(3):589--619, 2004.

\bibitem[HK17]{HoyoisKrishna}
Marc Hoyois and Amalendu Krishna.
\newblock Vanishing theorems for the negative {K}-theory of stacks.
\newblock {\em arXiv preprint arXiv:1705.02295}, 2017.

\bibitem[Hoy16]{HoyoisCdh}
Marc Hoyois.
\newblock Cdh descent in equivariant homotopy {K}-theory.
\newblock {\em arXiv preprint arXiv:1604.06410}, 2016.

\bibitem[HR17]{HallRydhPerfect}
Jack Hall and David Rydh.
\newblock Perfect complexes on algebraic stacks.
\newblock {\em Compositio Mathematica}, 153(11):2318--2367, 2017.

\bibitem[Kha16]{KhanThesis}
Adeel~A. Khan.
\newblock {\em Motivic homotopy theory in derived algebraic geometry}.
\newblock PhD thesis, Universit{\"a}t Duisburg-Essen, 2016.
\newblock Available at \url{https://www.preschema.com/thesis/}.

\bibitem[Kha19]{KhanLocalization}
Adeel~A. Khan.
\newblock The {M}orel--{V}oevodsky localization theorem in spectral algebraic
  geometry.
\newblock {\em Geom. Topol.}, 23(7):3647--3685, 2019.

\bibitem[KR18a]{KhanBlowup}
Adeel~A. Khan and David Rydh.
\newblock Virtual {C}artier divisors and blow-ups.
\newblock {\em arXiv preprint arXiv:1802.05702}, 2018.

\bibitem[KR18b]{KrishnaRavi}
Amalendu Krishna and Charanya Ravi.
\newblock Algebraic {K}-theory of quotient stacks.
\newblock {\em Annals of K-Theory}, 3(2):207--233, 2018.

\bibitem[KST18]{KerzStrunkTamme}
Moritz Kerz, Florian Strunk, and Georg Tamme.
\newblock Algebraic {K}-theory and descent for blow-ups.
\newblock {\em Inventiones mathematicae}, 211(2):523--577, 2018.

\bibitem[LT18]{LandTamme}
Markus Land and Georg Tamme.
\newblock On the {K}-theory of pullbacks.
\newblock {\em arXiv preprint arXiv:1808.05559}, 2018.

\bibitem[HTT]{HTT}
Jacob Lurie.
\newblock {\em Higher topos theory}.
\newblock Number 170. Princeton University Press, 2009.

\bibitem[HA]{HA-20170918}
Jacob Lurie.
\newblock Higher algebra.
\newblock {\em Preprint, available at
  \url{www.math.harvard.edu/~lurie/papers/HigherAlgebra.pdf}}, 2012.
\newblock Version of 2017-09-18.

\bibitem[SAG]{SAG-20180204}
Jacob Lurie.
\newblock Spectral algebraic geometry.
\newblock {\em Preprint, available at
  \url{www.math.harvard.edu/~lurie/papers/SAG-rootfile.pdf}}, 2016.
\newblock Version of 2018-02-03.

\bibitem[MV99]{MorelVoevodsky}
Fabien Morel and Vladimir Voevodsky.
\newblock $\mathbf{A}^1$-homotopy theory of schemes.
\newblock {\em Publications Math{\'e}matiques de l'IHES}, 90(1):45--143, 1999.

\bibitem[Orl92]{OrlovSOD}
Dmitri~Olegovich Orlov.
\newblock Projective bundles, monoidal transformations, and derived categories
  of coherent sheaves.
\newblock {\em Izvestiya Rossiiskoi Akademii Nauk. Seriya Matematicheskaya},
  56(4):852--862, 1992.

\bibitem[RG71]{RaynaudGruson}
Michel Raynaud and Laurent Gruson.
\newblock Critères de platitude et de projectivité: {T}echniques de
  ``platification'' d'un module.
\newblock {\em Inventiones mathematicae}, 13(1-2):1--89, 1971.

\bibitem[Tho87]{ThomasonEquivariant}
Robert~W Thomason.
\newblock Equivariant resolution, linearization, and {H}ilbert's fourteenth
  problem over arbitrary base schemes.
\newblock {\em Advances in Mathematics}, 65(1):16--34, 1987.

\bibitem[Tho93a]{ThomasonProjectiveBundle}
R.~W. Thomason.
\newblock Les {$K$}-groupes d'un fibr\'e projectif.
\newblock In {\em Algebraic {$K$}-theory and algebraic topology ({L}ake
  {L}ouise, {AB}, 1991)}, volume 407 of {\em NATO Adv. Sci. Inst. Ser. C Math.
  Phys. Sci.}, pages 243--248. Kluwer Acad. Publ., Dordrecht, 1993.

\bibitem[Tho93b]{ThomasonBlowup}
R.~W. Thomason.
\newblock Les {$K$}-groupes d'un sch\'ema \'eclat\'e et une formule
  d'intersection exc\'edentaire.
\newblock {\em Invent. Math.}, 112(1):195--215, 1993.

\bibitem[TT90]{ThomasonTrobaugh}
Robert~W Thomason and Thomas Trobaugh.
\newblock Higher algebraic {K}-theory of schemes and of derived categories.
\newblock In {\em The Grothendieck Festschrift}, pages 247--435. Springer,
  1990.

\bibitem[TV08]{HAG2}
Bertrand To{\"e}n and Gabriele Vezzosi.
\newblock Homotopical algebraic geometry {II}: {G}eometric stacks and
  applications.
\newblock {\em Mem. Amer. Math. Soc.}, 193(902):x+224, 2008.

\bibitem[Voe10a]{VoevodskyCD}
Vladimir Voevodsky.
\newblock Homotopy theory of simplicial sheaves in completely decomposable
  topologies.
\newblock {\em Journal of pure and applied algebra}, 214(8):1384--1398, 2010.

\bibitem[Voe10b]{VoevodskyCdh}
Vladimir Voevodsky.
\newblock Unstable motivic homotopy categories in {N}isnevich and
  cdh-topologies.
\newblock {\em Journal of Pure and applied Algebra}, 214(8):1399--1406, 2010.

\bibitem[Wei89]{WeibelKH}
Charles~A Weibel.
\newblock Homotopy algebraic {K}-theory.
\newblock {\em Algebraic K-theory and algebraic number theory (Honolulu, HI,
  1987)}, 83:461--488, 1989.

\end{thebibliography}
}

\end{document}